\documentclass{amsart} 
\usepackage{latexsym}
\usepackage{amsfonts}
\usepackage{amssymb}
\usepackage{amsmath}
\usepackage{amscd}
\usepackage{graphicx}
\usepackage{eucal}
\usepackage{amsthm}
\usepackage[all]{xy}
\usepackage{pdfsync}

\newcommand{\mb}{\mathbf}

\newcommand{\N}{\mathbb{N}}
\newcommand{\m}[1]{\mathcal{#1}}

\begin{document}
\newtheorem{thm}{Theorem}[section]
\newtheorem{prop}[thm]{Proposition}
\newtheorem{lem}[thm]{Lemma}
\newtheorem{cor}[thm]{Corollary}

\newtheoremstyle{example}{\topsep}{\topsep}%
     {}
     {}
     {\bfseries}
     {.}
     {2pt}
     {\thmname{#1}\thmnumber{ #2}\thmnote{ #3}}

   \theoremstyle{example}
	\newtheorem{conv}[thm]{Conventions}
   \newtheorem{nota}[thm]{Notation}
   \newtheorem{example}[thm]{Example}
	\newtheorem{Defi}[thm]{Definition}
   \newtheorem{rem}[thm]{Remark}

   \title[Operads]{Operads with general groups of equivariance, and some 2-categorical aspects of operads in $\mb{Cat}$}

\author{Alexander S. Corner}
\address{
School of Mathematics and Statistics,
University of Sheffield,
Sheffield, UK, S3 7RH
}
\email{a.s.corner@sheffield.ac.uk}
\author{Nick Gurski}
\address{
School of Mathematics and Statistics,
University of Sheffield,
Sheffield, UK, S3 7RH
}
\email{nick.gurski@sheffield.ac.uk}
\keywords{operad, 2-category, pseudo-commutativity}
\subjclass{18D05, 18D10, 18D50}

\begin{abstract}
We give a definition of an operad with general groups of equivariance suitable for use in any symmetric monoidal category with appropriate colimits.  We then apply this notion to study the 2-category of algebras over an operad in $\mb{Cat}$.  We show that any operad is finitary, that an operad is cartesian if and only if the group actions are nearly free (in a precise fashion), and that the existence of a pseudo-commutative structure largely depends on the groups of equivariance.  We conclude by showing that the operad for braided strict monoidal categories has two canonical pseudo-commutative structures.
\end{abstract}

\maketitle

\tableofcontents
	
\section*{Introduction}

Operads were defined by May \cite{maygeom} in the early 70's to provide a convenient tool to approach problems in algebraic topology, notably the question of when a space $X$ admits an $n$-fold delooping $Y$ so that $X \simeq \Omega^{n}Y$.  An operad, like an algebraic theory \cite{lawvere-thesis}, is something like a presentation for a monad or algebraic structure.  The theory of operads has seen great success, and we would like to highlight two reasons.  First, operads can be defined in any suitable symmetric monoidal category, so that there are operads of topological spaces, of chain complexes, of simplicial sets, and of categories, to name a few examples.  Moreover, symmetric (lax) monoidal functors carry operads to operads, so we can use operads in one category to understand objects in another via transport by such a functor.  Second, operads in a fixed category are highly flexible tools.  In particular, the categories listed above all have some inherent notion of ``homotopy equivalence'' which is weaker than that of isomorphism, so we can study operads which are equivalent but not isomorphic.  These tend to have algebras which have similar features in an ``up-to-homotopy'' sense but very different combinatorial or geometric properties arising from the fact that different objects make up these equivalent but not isomorphic operads.

Operads in the category $\mb{Cat}$ of small categories have a unique flavor arising from the fact that $\mb{Cat}$ is not just a category but a 2-category.  These 2-categorical aspects have not been widely treated in the literature, although a few examples can be found.  Lack \cite{lack-cod} mentions braided $\mb{Cat}$-operads (the reader new to braided operads should refer to the work of Fiedorowicz \cite{fie-br}) in his work on coherence for 2-monads, and Batanin \cite{bat-eh} uses lax morphisms of operads in $\mb{Cat}$ in order to define the notion of an internal operad.  But aside from a few appearances, the basic theory of operads in $\mb{Cat}$ and their 2-categorical properties seems missing.  This paper was partly motivated by the need for such a theory to be explained from the ground up.

There were two additional motivations for the work in this paper.  In thinking about coherence for monoidal functors, the first author was led to a general study of algebras for multicategories internal to $\mb{Cat}$.  These give rise to 2-monads (or perhaps pseudomonads, depending on how the theory is set up), and checking abstract properties of these 2-monads prompts one to consider the simpler case of operads in $\mb{Cat}$ instead of multicategories.  The other motivation was from the second author's attempt to understand the interplay between operads in $\mb{Cat}$, operads in $\mb{Top}$, and the passage from (bi)permutative categories to $E_{\infty}$ (ring) spaces.  The first of these motivations raised the issue of when operads in $\mb{Cat}$ are cartesian, while the second led us to consider when an operad in $\mb{Cat}$ possesses a pseudo-commutative structure.

While considering how to best tackle a general discussion of operads in $\mb{Cat}$, it became clear that restricting attention to the two most commonly used types of operads, symmetric and non-symmetric operads, was both short-sighted and unnecessary.  Many theorems apply to both kinds of operads at once, with the difference in proofs being negligible; in fact, most of the arguments which applied to the symmetric case seemed to apply to the case of braided operads as well.   This led us to the notion of an action operad $\mb{G}$, and then to a definition of $\mb{G}$-operads.  In essence, this is merely the general notion of what it means for an operad $P = \{ P(n) \}_{n \in \N}$ to have groups of equivariance $\mb{G} = \{ G(n) \}_{n \in \N}$ such that $G(n)$ acts on $P(n)$.  Choosing different natural families of groups $\mb{G}$, we recover known variants of the definition of operad. \\ \begin{center}
\begin{tabular}{c|c}
Groups $\mb{G}$ & Type of operad  \\ \hline
Terminal groups & Non-symmetric operad \\
Symmetric groups & Symmetric operad \\
Braid groups & Braided operad \\
\end{tabular} \\ \end{center}
These definitions have appeared, with minor variations, in two sources of which we are aware.  In Wahl's thesis \cite{wahl-thesis}, the essential definitions appear but not in complete generality as she requires a surjectivity condition.  Zhang \cite{zhang-grp} also studies these notions\footnote{Zhang calls our action operad a \textit{group operad}.  We dislike this terminology as it seems to imply that we are dealing with an operad in the category of groups, which is not the case unless all of the maps $\pi_{n}:G(n) \rightarrow \Sigma_{n}$ are zero maps.}, once again in the context of homotopy theory, but requires the  superfluous condition that $e_{1} = \textrm{id}$ (see Lemma \ref{calclem}).

This paper consists of the following.  In Section 1, we give the definition of an action operad $\mb{G}$ and a $\mb{G}$-operad.  We develop this definition abstractly so as to apply it in any suitable symmetric monoidal category.   It is standard to express operads as monoids in a particular functor category using a composition tensor product.  In order to show that our $\mb{G}$-operads fit into this philosophy, we must work abstractly and use the calculus of coends together with the Day convolution product \cite{day-thesis}.  The reader uninterested in these details can happily skip them, although we find the route taken here to be quite satisfactory in justifying the axioms for an action operad $\mb{G}$ and the accompanying notion of $\mb{G}$-operad.  Many of our calculations are generalizations of those appearing in work of Kelly \cite{kelly-op}, although there are slight differences in flavor between the two treatments.

Section 2 works through the basic 2-categorical aspects of operads in $\mb{Cat}$.  We explain how every operad gives rise to a 2-monad, and show that all of the various 1-cells between algebras of the associated 2-monad correspond to the obvious sorts of 1-cells one might define between algebras over an operad in $\mb{Cat}$.  Similarly, we show that the algebra 2-cells, using the 2-monadic approach, correspond to the obvious notion of transformation one would define using the operad.

Section 3 studies three basic 2-categorical properties of an operad, namely the property of being finitary, the property of being 2-cartesian, and the coherence property.  The first of these always holds, as a simple calculation shows.  The second of these turns out to be equivalent to the action of $G(n)$ on $P(n)$ being free for all $n$, at least up to a certain kernel.  In particular, our characterization clearly shows that every non-symmetric operad is 2-cartesian, and that a symmetric operad is 2-cartesian if and only if the symmetric group actions are all free.  (It is useful to note that a 2-monad on $\mb{Cat}$ is 2-cartesian if and only if the underlying monad on the category of small categories is cartesian in the usual sense as the (strict) 2-pullback of a diagram is the same as its pullback.)  The third property is also easily shown to hold for any $\mb{G}$-operad on $\mb{Cat}$ using a factorization system argument due to Power \cite{power-gen}.

Section 4 then goes on to study the question of when a $\mb{G}$-operad $P$ admits a pseudo-commutative structure.  Such a structure provides the 2-category of algebras with a richer structure that includes well-behaved notions of tensor product, internal hom, and multilinear map that fit together much as the analogous notions do in the category of vector spaces.  When $P$ is contractible (i.e., each $P(n)$ is equivalent to the terminal category), this structure can be obtained from a collection of elements $t_{m,n} \in G(mn)$ satisfying certain properties.  In particular, we show that every contractible symmetric operad is pseudo-commutative, and we prove that there exist such elements $t_{m,n} \in Br_{mn}$ so that every contractible braided operad is pseudo-commutative as well (in fact in two canonical ways).  Thus Section 4 can be seen as a continuation, in the operadic context, of the work in \cite{HP}, and in particular the ``geometric'' proof of the existence of a pseudo-commutative structure for braided strict monoidal categories demonstrates the power of being able to change the groups of equivariance.

The authors would like to thank John Bourke, Martin Hyland, Tom Leinster, and Peter May for various conversations which led to this paper.  While conducting this research, the second author was supported by an EPSRC Early Career Fellowship.

\section{Operads in symmetric monoidal categories}

In this section, we will explore the general definition of an operad $P$ which is equipped with groups of equivariance $G(n)$.  The group $G(n)$ will act on the right on the object $P(n)$, and the operad structure of $P$ will be required to respect this action.  For certain choices of the groups $G(n)$, we will recover standard notions of operads such as symmetric operads, non-symmetric operads, and braided operads.  The definitions here will, unless otherwise stated, apply in any symmetric monoidal category $\mathcal{V}$ in which the functors $X \otimes -, - \otimes X$ preserve colimits for every object $X \in \mathcal{V}$.

\begin{conv}
We adopt the following conventions throughout.
\begin{itemize}
\item $\Sigma_{n}$ is the symmetric group on $n$ letters, and $Br_{n}$ is the braid group on $n$ strands.
\item For a group $G$, a right $G$-action on a set $X$ will be denoted $(x,g) \mapsto x \cdot g$.  We will use both $\cdot$ and concatenation to represent multiplication in a group.
\item The symbol $e$ will generically represent an identity element in a group.  If we have a set of groups $\{ G(n) \}_{n \in \N}$ indexed by the natural numbers, then $e_{n}$ is the identity element in $G(n)$.
\item We will often be interested in elements of a product of the form
\[
A \times B_{1} \times \cdots \times B_{n} \times C
\]
(or similar, for example without $C$).  We will write elements of this set as $(a; b_{1}, \ldots, b_{n}; c)$, and in the case that we need equivalence classes of such elements they will be written as $[a; b_{1}, \ldots, b_{n}; c]$.
\end{itemize}
\end{conv}

We begin with the basic definitions.

\begin{Defi}
An \textit{operad} $O$ (in the category of sets) consists of
\begin{itemize}
\item sets $O(n)$ for each natural number $n$,
\item an element $\textrm{id} \in O(1)$, and
\item functions
\[
\mu: O(n) \times O(k_{1}) \times \cdots \times O(k_{n}) \rightarrow O(k_{1} + \cdots + k_{n}),
\]
\end{itemize}
satisfying the following two axioms.
\begin{enumerate}
\item The element $\textrm{id} \in O(1)$ is a two-sided unit for $\mu$ in the sense that
\[
\begin{array}{rcl}
\mu(\textrm{id}; x) & = & x \\
\mu(x; \textrm{id}, \ldots, \textrm{id}) & = & x
\end{array}
\]
for any $x \in O(n)$.
\item The functions $\mu$ (called operadic multiplication or operadic composition) are associative in the sense that the diagram below commutes.
\[
\xy
(0,0)*+{\scriptstyle O(n) \times O(k_{1}) \times \cdots \times O(k_{n}) \times O(l_{1,1}) \times \cdots \times O(l_{1},k_{1}) \times \cdots \times O(l_{n,1}) \times \cdots \times O(l_{n},k_{n})} ="00";
(0,-50)*+{\scriptstyle O(k_{1} + \cdots + k_{n}) \times O(l_{1,1}) \times \cdots \times O(l_{1},k_{1}) \times \cdots \times O(l_{n,1})\times \cdots \times O(l_{n},k_{n})} ="02";
(55,-10)*+{\scriptstyle O(n) \times \prod_{i=1}^n O(k_{1}) \times O(l_{i,1}) \times \cdots \times O(l_{i}, k_{i}) } ="20";
(55,-25)*+{\scriptstyle O(n) \times O(\sum l_{1,-}) \times \cdots \times O(\sum l_{n,-})} ="21";
(55, -40)*+{\scriptstyle  O(\sum l_{-,-})} ="22";
{\ar_{\scriptstyle \mu \times 1} "00" ; "02"};
{\ar_{\mu} "02" ; "22"};
{\ar^{\cong} "00" ; "20"};
{\ar^{1 \times \prod \mu} "20" ; "21"};
{\ar^{\mu} "21" ; "22"};
\endxy
\]
\end{enumerate}
\end{Defi}

\begin{rem}
\begin{enumerate}
\item One can change from operads in $\mb{Sets}$ to operads in another symmetric monoidal category $\mathcal{V}$ by requiring each $O(n)$ to be an object of $\mathcal{V}$ and replacing all instances of cartesian product with the appropriate tensor product in $\mathcal{V}$.  This includes replacing the element $\textrm{id} \in O(1)$ with a map $I \rightarrow O(1)$ from the unit object of $\mathcal{V}$ to $O(1)$.
\item Every operad has an underlying \textit{collection} which consists of the natural number-indexed set $\{ O(n) \}_{n \in \N}$, but without a chosen identity element or composition maps.  The category of collections is a presheaf category (in this case on the discrete category of natural numbers), and we will equip it with a monoidal structure in which monoids are precisely operads in Theorem \ref{operad=monoid}.
\end{enumerate}
\end{rem}

One is intended to think that $x \in O(n)$ is a function with $n$ inputs and a single output, as below.
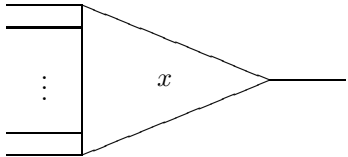
\[
\xy
{\ar@{-} (0,0)*{}; (25,-10)*{} };
{\ar@{-} (0,-20)*{}; (25,-10)*{} };
{\ar@{-} (0,0)*{}; (0,-20)*{} };
{\ar@{-} (25,-10)*{}; (35,-10)*{} };
{\ar@{-} (0,0)*{}; (-10,0)*{} };
{\ar@{-} (0,-3)*{}; (-10,-3)*{} };
{\ar@{-} (0,-17)*{}; (-10,-17)*{} };
{\ar@{-} (0,-20)*{}; (-10,-20)*{} };
(11,-10)*{x}; (-5,-10)*{\vdots}
\endxy
\]
Operadic composition is then a generalization of function composition, with the pictorial representation below being $\mu(x; y_{1}, y_{2})$ for $\mu:O(2) \times O(2) \times O(3) \rightarrow O(5)$.
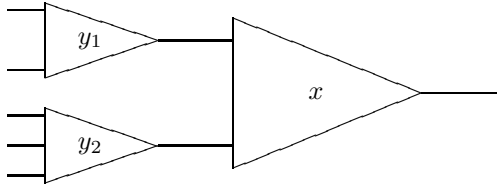
\[
\xy
{\ar@{-} (0,0)*{}; (25,-10)*{} };
{\ar@{-} (0,-20)*{}; (25,-10)*{} };
{\ar@{-} (0,0)*{}; (0,-20)*{} };
{\ar@{-} (25,-10)*{}; (35,-10)*{} };
{\ar@{-} (0,-3)*{}; (-10,-3)*{} };
{\ar@{-} (0,-17)*{}; (-10,-17)*{} };
(11,-10)*{x};
{\ar@{-} (-25,2)*{}; (-10,-3)*{} };
{\ar@{-} (-25,-8)*{}; (-10,-3)*{} };
{\ar@{-} (-25,2)*{}; (-25,-8)*{} };
{\ar@{-} (-25,1)*{}; (-30,1)*{} };
{\ar@{-} (-30,-7)*{}; (-25,-7)*{} };
(-19,-3)*{y_{1}};
{\ar@{-} (-25,-12)*{}; (-10,-17)*{} };
{\ar@{-} (-25,-22)*{}; (-10,-17)*{} };
{\ar@{-} (-25,-12)*{}; (-25,-22)*{} };
{\ar@{-} (-25,-13)*{}; (-30,-13)*{} };
{\ar@{-} (-25,-17)*{}; (-30,-17)*{} };
{\ar@{-} (-25,-21)*{}; (-30,-21)*{} };
(-19,-17)*{y_{2}};
\endxy
\]

\begin{example}
The canonical example of an operad is the \textit{symmetric operad} which we write as $\Sigma$.  The set $\Sigma(n)$ is the set of elements of the symmetric group $\Sigma_{n}$.  The identity element $\textrm{id} \in \Sigma(1)$ is just the identity permutation on a one-element set.  Operadic composition in $\Sigma$ will then be given by a function
\[
\Sigma(n) \times \Sigma(k_{1}) \times \cdots \times \Sigma(k_{n}) \rightarrow \Sigma(k_{1} + \cdots + k_{n})
\]
which takes permutations $\sigma \in \Sigma_{n}, \tau_{i} \in \Sigma_{k_{i}}$ and produces the following permutation in $\Sigma_{k_{1} + \cdots + k_{n}}$.  First we form the block sum permutation $\tau_{1} \oplus \cdots \oplus \tau_{n}$ which permutes the first $k_{1}$ elements according to $\tau_{1}$, the next $k_{2}$ elements according to $\tau_{2}$ and so on; this is an element of $\Sigma_{k_{1} + \cdots + k_{n}}$.  Then we take the permutation $\sigma^+ \in \Sigma_{k_{1} + \cdots + k_{n}}$ which permutes the $n$ different blocks $1$ through $k_{1}$, $k_{1}+1$ through $k_{1} + k_{2}$, and so on, according to the permutation $\sigma \in \Sigma_{n}$.  Operadic composition in $\Sigma$ is then given by the formula
\[
\mu(\sigma; \tau_{1}, \ldots, \tau_{n}) = \sigma^+ \cdot (\tau_{1} \oplus \cdots \oplus \tau_{n}).
\]
Below we have drawn the permutation for the composition
\[
\mu:\Sigma(3) \times \Sigma(2) \times \Sigma(4) \times \Sigma(3) \rightarrow \Sigma(9)
\]
evaluated on the element $\Big( (123); (12), (12)(34), (13) \Big)$.
\[
\xy
{\ar@{-} (0,0)*{}; (5,-5)*{} };
{\ar@{-} (5,0)*{}; (0,-5)*{} };
{\ar@{-} (12,0)*{}; (17,-5)*{} };
{\ar@{-} (17,0)*{}; (12,-5)*{} };
{\ar@{-} (22,0)*{}; (27,-5)*{} };
{\ar@{-} (27,0)*{}; (22,-5)*{} };
{\ar@{-} (34,0)*{}; (44,-5)*{} };
{\ar@{-} (39,0)*{}; (39,-5)*{} };
{\ar@{-} (44,0)*{}; (34,-5)*{} };
{\ar@{-} (0,-5)*{}; (17,-13)*{} };
{\ar@{-} (5,-5)*{}; (22,-13)*{} };
{\ar@{-} (12,-5)*{}; (29,-13)*{} };
{\ar@{-} (17,-5)*{}; (34,-13)*{} };
{\ar@{-} (22,-5)*{}; (39,-13)*{} };
{\ar@{-} (27,-5)*{}; (44,-13)*{} };
{\ar@{-} (34,-5)*{}; (0,-13)*{} };
{\ar@{-} (39,-5)*{}; (5,-13)*{} };
{\ar@{-} (44,-5)*{}; (10,-13)*{} };
\endxy
\]
Note that $(12)(34) \in \Sigma(4)$ is actually $\mu(e_{2}; (12), (12))$, where $e_{2} \in \Sigma_{2}$ is the identity permutation.  Using this and operad associativity, one can easily check that
\[
\mu \Big( (123); (12), (12)(34), (13) \Big) = \mu \Big( (1234); (12), (12), (12), (13) \Big),
\]
where now the composition on the right side uses the function
\[
\mu:\Sigma(4) \times \Sigma(2) \times \Sigma(2) \times \Sigma(2) \times \Sigma(3) \rightarrow \Sigma(9).
\]
This equality is obvious using the picture above, but verifiable directly using only the algebra of the symmetric operad.
\end{example}

The definition we have given above is for what some might call a \textit{plain} or \textit{non-symmetric operad}.  In many applications, something more sophisticated is required.

\begin{Defi}
A \textit{symmetric operad} consists of
\begin{itemize}
\item an operad $O$ and
\item for each $n$, a right $\Sigma_{n}$-action on $O(n)$,
\end{itemize}
satisfying the following axioms.
\[
\begin{array}{rcl}
\mu(x; y_{1} \cdot \tau_{1}, \ldots, y_{n} \cdot \tau_{n}) & = & \mu(x; y_{1}, \ldots, y_{n}) \cdot (\tau_{1} \oplus \cdots \oplus \tau_{n}) \\
\mu(x \cdot \sigma; y_{1}, \ldots, y_{n}) & =  & \mu(x; y_{\sigma^{-1}(1)}, \ldots, y_{\sigma^{-1}(n)}) \cdot \sigma^{+}
\end{array}
\]
For the above equations to make sense, we must have
\begin{itemize}
\item $x \in O(n)$,
\item $y_{i} \in O(k_{i})$ for $i=1, \ldots, n$,
\item $\tau_{i} \in \Sigma_{k_{i}}$, and
\item $\sigma \in \Sigma_{n}$.
\end{itemize}
\end{Defi}

The operad $\Sigma$ detailed above is a symmetric operad, with the symmetric group action being given by right multiplication.  We leave it to the reader to check the axioms, but in each case they are entirely straightforward.  In the original topological applications \cite{maygeom}, symmetric operads were the central figures while plain operads were generally not as useful.  A further kind of operad was studied by Fiedorowicz in \cite{fie-br}; we give the definition below in analogy with that for symmetric operads, with interpretation to follow afterwards to make it entirely rigorous.  We do this to emphasize the key features that we will generalize in Definition \ref{aoperad}.

\begin{Defi}\label{broperad}
A \textit{braided operad} consists of
\begin{itemize}
\item an operad $O$ and
\item for each $n$, a right action of the $n$th braid group $Br_{n}$ on $O(n)$,
\end{itemize}
satisfying the following axioms.
\[
\begin{array}{rcl}
\mu(x; y_{1} \cdot \tau_{1}, \ldots, y_{n} \cdot \tau_{n}) & = & \mu(x; y_{1}, \ldots, y_{n}) \cdot (\tau_{1} \oplus \cdots \oplus \tau_{n}) \\
\mu(x \cdot \sigma; y_{1}, \ldots, y_{n}) & =  & \mu(x; y_{\sigma^{-1}(1)}, \ldots, y_{\sigma^{-1}(n)}) \cdot \sigma^{+}
\end{array}
\]
For the above equations to make sense, we must have
\begin{itemize}
\item $x \in O(n)$,
\item $y_{i} \in O(k_{i})$ for $i=1, \ldots, n$,
\item $\tau_{i} \in Br_{k_{i}}$, and
\item $\sigma \in Br_{n}$.
\end{itemize}
\end{Defi}

In order to make sense of this definition, we must define $\tau_{1} \oplus \cdots \oplus \tau_{n}$ and $\sigma^{+}$ in the context of braids.  The first is the block sum in the obvious sense:  given $n$ different braids on $k_{1}, \ldots, k_{n}$ strands, respectively, we form a new braid on $k_{1} + \cdots + k_{n}$ strands by taking a disjoint union where the braid $\tau_{i}$ is to the left of $\tau_{j}$ if $i < j$.  The braid $\sigma^{+}$ is obtained by replacing the $i$th strand with $k_{i}$ consecutive strands, all of which are braided together according to $\sigma$.  Finally, the notation $\sigma^{-1}(i)$ should be read as $\pi(\sigma)^{-1}(i)$, where $\pi:Br_{n} \rightarrow \Sigma_{n}$ is the underlying permutation map.

We require one final preparatory definition.

\begin{Defi}
Let $O,O'$ be operads.  Then an \textit{operad map} $f:O \rightarrow O'$ consists of functions $f_{n}:O(n) \rightarrow O'(n)$ for each natural number such that the following axioms hold.
\[
\begin{array}{rcl}
f(\textrm{id}_{O}) & = & \textrm{id}_{O'} \\
f(\mu^{O}(x; y_{1}, \ldots, y_{n})) & = & \mu^{O'}(f(x); f(y_{1}), \ldots, f(y_{n}))
\end{array}
\]
\end{Defi}

\begin{conv}
In the above definition and below, we adopt the convention that if an equation requires using operadic composition in more than one operad, we will indicate this by a superscript on each instance of $\mu$ unless it is entirely clear from context.
\end{conv}

\begin{example}
One can form an operad $Br$ where $Br(n)$ is the underlying set of the $n$th braid group, $Br_{n}$.  This is done in much the same way as we did for the symmetric operad, and the collection of maps $\pi_{n}:Br_{n} \rightarrow \Sigma_{n}$ giving the underlying permutations constitutes an operad map $Br \rightarrow \Sigma$.
\end{example}

One should note that the axioms for symmetric and braided operads each use the fact that the groups of equivariance themselves form an operad.  This is what we call an action operad.

\begin{Defi}\label{aoperad}
An \textit{action operad} $\mb{G}$ consists of
\begin{itemize}
\item an operad $G = \{ G(n) \}_{n \in \N}$ in the category of sets such that each $G(n)$ is equipped with the structure of a group and
\item a map $\pi:G \rightarrow \Sigma$ which is simultaneously a map of operads and a group homomorphism $\pi_{n}:G(n) \rightarrow \Sigma_{n}$ for each $n$
\end{itemize}
such that
\[
\mu(g; f_1, \ldots f_n)  \mu(g'; f_1', \ldots, f_n') = \mu (gg'; f_{\pi(g')(1)}f_{1}', \ldots, f_{\pi(g')(n)}f_{n}')
\]
holds in the group $G(k_{1} + \cdots + k_{n})$, provided both sides make sense.  This occurs precisely when
\begin{itemize}
\item $g, g' \in G(n)$,
\item $f_{i} \in G(k_{\pi(g')^{-1}(i)})$, and
\item $f_{i}' \in G(k_{i})$.
\end{itemize}
\end{Defi}

\begin{rem}
\begin{itemize}
\item The final axiom is best explained using the operad $\Sigma$ of symmetric groups.  Reading symmetric group elements as permutations from top to bottom, below is a pictorial representation of the final axiom for the map $\mu:\Sigma_{3} \times \Sigma_{2} \times \Sigma_{2} \times \Sigma_{2} \rightarrow \Sigma_{6}.$
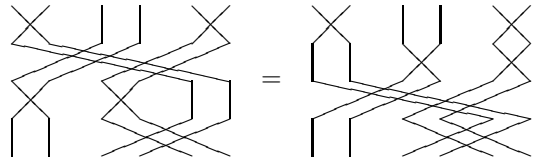
\[
\xy
{\ar@{-} (0,0)*{}; (5,-5)*{} };
{\ar@{-} (5,-5)*{}; (29,-10)*{} };
{\ar@{-} (5,0)*{}; (0,-5)*{} };
{\ar@{-} (0,-5)*{}; (24,-10)*{} };
{\ar@{-} (12,0)*{}; (12,-5)*{} };
{\ar@{-} (12,-5)*{}; (0,-10)*{} };
{\ar@{-} (17,0)*{}; (17,-5)*{} };
{\ar@{-} (17,-5)*{}; (5,-10)*{} };
{\ar@{-} (24,0)*{}; (29,-5)*{} };
{\ar@{-} (29,-5)*{}; (17,-10)*{} };
{\ar@{-} (29,0)*{}; (24,-5)*{} };
{\ar@{-} (24,-5)*{}; (12,-10)*{} };
{\ar@{-} (0,-10)*{}; (5,-15)*{} };
{\ar@{-} (5,-10)*{}; (0,-15)*{} };
{\ar@{-} (12,-10)*{}; (17,-15)*{} };
{\ar@{-} (17,-10)*{}; (12,-15)*{} };
{\ar@{-} (24,-10)*{}; (24,-15)*{} };
{\ar@{-} (29,-10)*{}; (29,-15)*{} };
{\ar@{-} (0,-15)*{}; (0,-20)*{} };
{\ar@{-} (5,-15)*{}; (5,-20)*{} };
{\ar@{-} (12,-15)*{}; (24,-20)*{} };
{\ar@{-} (17,-15)*{}; (29,-20)*{} };
{\ar@{-} (24,-15)*{}; (12,-20)*{} };
{\ar@{-} (29,-15)*{}; (17,-20)*{} };
{\ar@{-} (40,0)*{}; (45,-5)*{} };
{\ar@{-} (45,0)*{}; (40,-5)*{} };
{\ar@{-} (52,0)*{}; (52,-5)*{} };
{\ar@{-} (57,0)*{}; (57,-5)*{} };
{\ar@{-} (64,0)*{}; (69,-5)*{} };
{\ar@{-} (69,0)*{}; (64,-5)*{} };
{\ar@{-} (40,-5)*{}; (40,-10)*{} };
{\ar@{-} (45,-5)*{}; (45,-10)*{} };
{\ar@{-} (52,-5)*{}; (57,-10)*{} };
{\ar@{-} (57,-5)*{}; (52,-10)*{} };
{\ar@{-} (64,-5)*{}; (69,-10)*{} };
{\ar@{-} (69,-5)*{}; (64,-10)*{} };
{\ar@{-} (40,-10)*{}; (64,-15)*{} };
{\ar@{-} (45,-10)*{}; (69,-15)*{} };
{\ar@{-} (52,-10)*{}; (40,-15)*{} };
{\ar@{-} (57,-10)*{}; (45,-15)*{} };
{\ar@{-} (64,-10)*{}; (52,-15)*{} };
{\ar@{-} (69,-10)*{}; (57,-15)*{} };
{\ar@{-} (40,-15)*{}; (40,-20)*{} };
{\ar@{-} (45,-15)*{}; (45,-20)*{} };
{\ar@{-} (52,-15)*{}; (64,-20)*{} };
{\ar@{-} (57,-15)*{}; (69,-20)*{} };
{\ar@{-} (64,-15)*{}; (52,-20)*{} };
{\ar@{-} (69,-15)*{}; (57,-20)*{} };
(34.5,-10)*+{=};
(4.5,-25)*{\scriptstyle \mu\big((23);(12),(12), \textrm{id}\big) \cdot \mu\big((132); (12), \textrm{id}, (12)\big) };
(64.5,-25)*{\scriptstyle \mu\big((23)\cdot (132); \textrm{id} \cdot (12), (12) \cdot \textrm{id}, (12) \cdot (12)\big)} ;
\endxy
\]
\item Our definition of an action operad is the same as that appearing in Wahl's thesis \cite{wahl-thesis}, but without the condition that each $\pi_{n}$ is surjective.  It is also the same as that appearing in work of Zhang \cite{zhang-grp}, although we prove later (see Lemma \ref{calclem}) that the condition $e_{1} = \textrm{id}$ in Zhang's definition follows from the rest of the axioms.
\end{itemize}
\end{rem}

\begin{example}
\begin{enumerate}
\item There are two trivial examples of action operads.  The first is the symmetric operad  $\mb{G} = \mb{\Sigma}$ with the identity map; this is the terminal object in the category of action operads (see Definition \ref{mapaop}).  The second is $\mb{G} = \mb{T}$ consisting of the terminal groups $T(n) = *$.  Here the $\pi_{n}$'s are given by the inclusion of identity elements; this is the initial object in the category of action operads.
\item Two less trivial examples are given by the braid groups, $\mb{G} = \mb{Br}$, and the ribbon braid groups, $\mb{G} = \mb{RBr}$.  (A ribbon braid is given, geometrically, as a braid with strands replaced by ribbons in which we allow full twists.  The actual definition of the ribbon braid groups is as the fundamental group of a configuration space in which points have labels in the circle, $S^{1}$; see \cite{sal-wahl}.)  In each case, the homomorphism $\pi$ is given by taking underlying permutations, and the operad structure is given geometrically by using the procedure explained after Definition \ref{broperad}.  We refer the reader to \cite{fie-br} for more information about braided operads, and to \cite{sal-wahl, wahl-thesis} for information about the ribbon case.
\end{enumerate}
\end{example}

Action operads are themselves the objects of a category, $\mb{AOp}$.  The morphisms of this category are defined below.
\begin{Defi}\label{mapaop}
A \textit{map of action operads} $f: \mb{G} \rightarrow \mb{G}'$ consists of a map $f:G \rightarrow G'$ of the underlying operads such that
\begin{enumerate}
\item $\pi^{G'} \circ f = \pi^{G}$ (i.e., $f$ is a map of operads over $\Sigma$) and
\item each $f_{n}:G(n) \rightarrow G'(n)$ is a group homomorphism.
\end{enumerate}
\end{Defi}


Just as we had the definitions of operad, symmetric operad, and braided operad, we now come to the general definition of a $\mb{G}$-operad, where $\mb{G}$ is an action operad.

\begin{Defi}
Let $\mb{G}$ be an action operad.  A \textit{$\mb{G}$-operad} $P$ (in $\mb{Sets}$) consists of
\begin{itemize}
\item an operad $P$ in $\mb{Sets}$ and
\item for each $n$, an action $P(n) \times G(n) \rightarrow P(n)$ of $G(n)$ on $P(n)$
\end{itemize}
such that the following two equivariance axioms hold.
\[
\begin{array}{c}
\mu^{P}(x; y_{1} \cdot g_{1}, \ldots, y_{n} \cdot g_{n}) =\mu^{P}(x; y_{1}, \ldots, y_{n}) \cdot \mu^{G}(e; g_{1}, \ldots, g_{n})  \\
\mu^{P}(x \cdot g; y_{1}, \ldots, y_{n})  =  \mu^{P}(x; y_{\pi(g)^{-1}(1)}, \ldots, y_{\pi(g)^{-1}(n)}) \cdot \mu^{G}(g; e_{1}, \ldots, e_{n})
\end{array}
\]
\end{Defi}

\begin{example}
\begin{enumerate}
\item Let $\mb{T}$ denote the terminal operad in $\mb{Sets}$ equipped with its unique action operad structure.  Then a $\mb{T}$-operad is just a non-symmetric operad in $\mb{Sets}$.
\item Let $\mb{\Sigma}$ denote the operad of symmetric groups with $\pi:\Sigma \rightarrow \Sigma$ the identity map.  Then a $\mb{\Sigma}$-operad is a symmetric operad in the category of sets.
\item Let $\mb{Br}$ denote the operad of braid groups with $\pi_{n}:Br_{n} \rightarrow \Sigma_{n}$ the canonical projection of a braid onto its underlying permutation.  Then a $\mb{Br}$-operad is a braided operad in the sense of Fiedorowicz \cite{fie-br}.
\end{enumerate}
\end{example}

\begin{rem}
It is possible to consider $\mb{G}$-operads in categories other than the category of sets.  In this case we still use the notion of an action operad given above, but then take the operad $P$ to have objects $P(n)$ which are the objects of some closed symmetric monoidal category $\mathcal{V}$.  We will rarely use anything that might require the closed structure as such, only the fact that the tensor product distributes over colimits in each variable.  This is a consequence of the fact that both $X \otimes -$ and $- \otimes X$ are left adjoints in the case of a closed symmetric monoidal category.  Thus while we set up the foundations using only operads in $\mb{Sets}$, the diligent reader can easily modify this theory for their closed symmetric monoidal category of choice.  In fact, we will use the same theory in $\mb{Cat}$ with its cartesian structure, noting only that the same arguments work in $\mb{Cat}$ with essentially no modification.
\end{rem}

\begin{Defi}
Let $\mb{G}$ be an action operad.  The category $\mb{G\mbox{-}Coll}$ of $\mb{G}$-collections has objects $X = \{ X(n) \}_{n \in \N}$ which consist of a set $X(n)$ for each natural number $n$ together with an action $X(n) \times G(n) \rightarrow X(n)$ of $G(n)$ on $X(n)$.  A morphism $f:X \rightarrow Y$ in $\mb{G\mbox{-}Coll}$ consists of a $G(n)$-equivariant map $f_{n}:X(n) \rightarrow Y(n)$ for each natural number $n$.
\end{Defi}

\begin{rem}
The definition of $\mb{G\mbox{-}Coll}$ does not require that $\mb{G}$ be an action operad, only that one has a natural number-indexed set of groups.  Given any such collection of groups $\{ G(n) \}_{n \in \N}$, we can form the category $\mathbb{G}$ whose objects are natural numbers and whose hom-sets are given by $\mathbb{G}(m,n) = \emptyset$ if $m \neq n$ and $\mathbb{G}(n,n) = G(n)$ (where composition and units are given by group multiplication and identity elements, respectively).  Then $\mb{G\mbox{-}Coll}$ is the presheaf category
\[
\hat{\mathbb{G}} = [\mathbb{G}^{\textrm{op}}, \mb{Sets}],
 \]
with the opposite category arising from our choice of right actions.  A key step in explaining how $\mb{G}$-operads arise as monoids in the category of $\mb{G}$-collections is to show that being an action operad endows $\mathbb{G}$ with a monoidal structure.
\end{rem}

\begin{Defi}
Let $\mb{G}$ be an action operad, and let $X, Y$ be $\mb{G}$-collections.  We define the $\mb{G}$-collection $X \circ Y$ to be
\[
X \circ Y (n) = \Big( \coprod_{k_{1} + \cdots + k_{r} = n} X(r) \times Y(k_{1}) \times \cdots \times Y(k_{r}) \Big) \times G(n) / \sim
\]
where the equivalence relation is generated by
\[
\begin{array}{rcl}
(xh; y_{1}, \ldots, y_{r}; g) & \sim & (x; y_{\pi(h)^{-1}(1)}, \ldots, y_{\pi(h)^{-1}(r)}; \mu(h; e, \ldots, e)g), \\
(x; y_{1}, \ldots, y_{r}; \mu(e; g_{1}, \ldots, g_{r})g) & \sim & (xe; y_{1}g_{1}, \ldots, y_{r}g_{r}; g).
\end{array}
\]
For the first relation above, we must have that the lefthand side is an element of
\[
X(r) \times Y(k_1) \times \cdots \times Y(k_r) \times G(n)\]
while the righthand side is an element of
\[
X(r) \times Y(k_{\pi(h)^{-1}(1)}) \times \cdots \times Y(k_{\pi(h)^{-1}(r)}) \times G(n);
\]
for the second relation, we must have $x \in X(r)$, $y_{i} \in Y(k_{i})$, $f \in G(r)$, $g_{i} \in G(k_{i})$, and $g \in G(n)$.  The right $G(n)$-action on $X \circ Y(n)$ is given by multiplication on the final coordinate.
\end{Defi}

We will now develop the tools to prove that the category $\mb{G\mbox{-}Coll}$ has a monoidal structure given by $\circ$, and that operads are the monoids therein.

\begin{thm}\label{operad=monoid}
Let $\mb{G}$ be an action operad.
\begin{enumerate}
\item The category $\mb{G\mbox{-}Coll}$ has a monoidal structure with tensor product given by $\circ$ and unit given by the collection $I$ with $I(n) = \emptyset$ when $n \neq 1$, and $I(1) = G(1)$ with the $\mb{G}$-action given by multiplication on the right.
\item The category $\mb{Mon}(\mb{G\mbox{-}Coll})$ of monoids in $\mb{G\mbox{-}Coll}$ is equivalent to the category of $\mb{G}$-operads with morphisms being those operad maps $P \rightarrow Q$ for which each $P(n) \rightarrow Q(n)$ is $G(n)$-equivariant.
\end{enumerate}
\end{thm}

While this theorem can be proven by direct calculation using the equivalence relation given above, such a proof is unenlightening.  Furthermore, we want to consider $\mb{G}$-operads in categories other than sets, so an element-wise proof might not apply.  Instead we now develop some general machinery that will apply to $\mb{G}$-operads in any cocomplete symmetric monoidal category in which each of  the functors $X \otimes -, - \otimes X$ preserve colimits (as is the case if the monoidal structure is closed).  This theory also demonstrates the importance of the final axiom in the definition of an action operad.  We begin with a calculational lemma.

\begin{lem}\label{calclem}
Let $\mb{G}$ be an action operad, and write $e_{n}$ for the unit element in the group $G(n)$.
\begin{enumerate}
\item In $G(1)$, the unit element $e_{1}$ for the group structure is equal to the identity element for the operad structure, $\textrm{id}$.
\item The equation
\[
\mu(e_{n}; e_{i_{1}}, \ldots, e_{i_{n}}) = e_{I}
\]
holds for any natural numbers $n, i_{j}, I = \sum i_{j}$.
\item The group $G(1)$ is abelian.
\end{enumerate}
\end{lem}
\begin{proof}
For the first claim, let $g \in G(1)$.  Then
\[
\begin{array}{rcl}
g & = & g \cdot e_{1} \\
& = & \mu(g; \textrm{id}) \cdot \mu(\textrm{id}; e_{1}) \\
& = & \mu(g \cdot \textrm{id}; \textrm{id} \cdot e_{1}) \\
& = & \mu(g \cdot \textrm{id}; \textrm{id}) \\
& = & g \cdot \textrm{id}
\end{array}
\]
using that $e_{1}$ is the unit element for the group structure, that $\textrm{id}$ is a two-sided unit for operad multiplication, and the final axiom for an action operad together with the fact that the only element of the symmetric group $\Sigma_{1}$ is the identity permutation.  Thus $g = g \cdot \textrm{id}$, so $\textrm{id} = e_{1}$.

For the second claim, write the operadic product as $\mu(e; \underline{e})$, and consider the square of this element. We have
\[
\begin{array}{rcl}
\mu(e; \underline{e}) \cdot \mu(e; \underline{e}) & = & \mu(e \cdot e; \underline{e} \cdot \underline{e}) \\
&= & \mu(e; \underline{e})
\end{array}
\]
where the first equality follows from the last action operad axiom together with the fact that $e$ gets mapped to the identity permutation; here $\underline{e} \cdot \underline{e}$ is the sequence $e_{i_{1}} \cdot e_{i_{1}}, \ldots, e_{i_{n}} \cdot e_{i_{n}}$.  Thus $\mu(e; \underline{e})$ is an idempotent element of the group $G(I)$, so must be the identity element $e_{I}$.

For the final claim, note that operadic multiplication $\mu:G(1) \times G(1) \rightarrow G(1)$ is a group homomorphism by the action operad axioms, and $\textrm{id} = e_{1}$ is a two-sided unit, so the Eckmann-Hilton argument shows that $\mu$ is actually group multiplication and that $G_{1}$ is abelian.
\end{proof}

Our construction of the monoidal structure on the category of $\mb{G}$-collections will require the Day convolution product \cite{day-thesis}.  This is a general construction which produces a monoidal structure on the category of presheaves $[\mathcal{V}^{\textrm{op}}, \mb{Sets}]$ from a monoidal structure on the category $\mathcal{V}$.  Since the category of $\mb{G}$-collections is the presheaf category $[\mathbb{G}^{\textrm{op}}, \mb{Sets}]$, we need to show that $\mathbb{G}$ has a monoidal structure.

\begin{prop}\label{Gmonoidal}
The action operad structure of $\mb{G}$ gives $\mathbb{G}$ a strict monoidal structure.
\end{prop}
\begin{proof}
The tensor product on $\mathbb{G}$ is given by addition on objects, with unit object 0.  The only thing to do is define the tensor product on morphisms and check naturality for the associativity and unit isomorphisms, which will both be identities.  On morphisms, $+$ must be given by a group homomorphism
\[
+:G(n) \times G(m) \rightarrow G(n+m),
\]
 and this is given by the formula
\[
+(g,h) = \mu(e_{2}; g,h).
\]
We need that $+$ is a group homomorphism, and the second part of Lemma \ref{calclem} shows that it preserves identity elements.  The final action operad axiom shows that it also preserves group multiplication since $\pi_{2}(e_{2}) = e_{2}$ (each $\pi_{n}$ is a group homomorphism) and therefore
\[
\begin{array}{rcl}
\Big(+(g,h)\Big) \cdot \Big(+(g',h')\Big) & = & \mu(e_{2}; g,h) \cdot \mu(e_{2}; g',h') \\
 & = & \mu(e_{2}e_{2}; gg', hh') \\
& = & +(gg',hh').
\end{array}
\]
We now write $+(g,h)$ as $g+h$.

For naturality of the associator, we must have $(f+g)+h = f+(g+h)$.  By the operad axioms for both units and associativity, the lefthand side is given by
\[
\begin{array}{rcl}
\mu(e_{2}; \mu(e_{2}; f,g), h) & = & \mu(e_{2}; \mu(e_{2}; f,g), \mu(\textrm{id};h)) \\
& = & \mu(\mu(e_{2}; e_{2}, \textrm{id}); f,g,h),
\end{array}
\]
while the righthand side is then
\[
\mu(e_{2}; f, \mu(e_{2}; g,h)) = \mu(\mu(e_{2}; \textrm{id}, e_{2}); f,g,h).
\]
By Lemma \ref{calclem}, both of these are equal to $\mu(e_{3}; f,g,h)$, proving associativity.  Naturality of the unit follows similarly, using $e_{0}$.
\end{proof}

Now that $\mathbb{G}$ has a monoidal structure, we get a monoidal structure on the category of $\mathbb{G}$-collections
\[
[\mathbb{G}^{\textrm{op}}, \mb{Sets}] = \hat{\mathbb{G}}
\]
using Day convolution, denoted $\star$.  Given collections $X, Y$, their convolution product $X \star Y$ is given by the coend formula
\[
X \star Y (k) = \int^{m,n \in \mathbb{G}} X(m) \times Y(n) \times \mathbb{G}(k, m+n)
\]
We refer the reader to \cite{day-thesis} for further details.  We do note, however, that the $n$-fold Day convolution product of a presheaf $Y$ with itself is given by the following coend formula.
\[
Y^{\star n}(k) = \int^{(k_{1}, \ldots, k_{n}) \in \mathbb{G}^{n}} Y(k_{1}) \times \cdots \times Y(k_{n}) \times \mathbb{G}(k, k_{1} + \cdots + k_{n})
\]
Computations with Day convolution will necessarily involve heavy use of the calculus of coends, and we refer the reader in need of a refresher course on coends to \cite{maclane-catwork}.  Our goal is to express the substitution tensor product as a coend just as in \cite{kelly-op}, and to do that we need one final result about the Day convolution product.

\begin{lem}\label{calclem2}
Let $\mb{G}$ be an action operad, let $Y \in \hat{\mathbb{G}}$, and let $k$ be a fixed natural number.  Then the assignment
\[
n \mapsto Y^{\star n}(k)
\]
can be given the structure of a functor $\mathbb{G} \rightarrow \mb{Sets}$.
\end{lem}
\begin{proof}
Since the convolution product is given by a coend, it is the universal object with maps
\[
Y(k_{1}) \times \cdots \times Y(k_{n}) \times \mathbb{G}(k, k_{1} + \cdots + k_{n}) \rightarrow Y^{\star n}(k)
\]
such that the following diagram commutes for every $g_{1} \in G(k_{1}), \ldots, g_{n} \in G(k_{n})$.
\[
\xy
{\ar   (0,0)*+{Y(k_{1}) \times \cdots \times Y(k_{n}) \times \mathbb{G}(k, k_{1} + \cdots + k_{n})}; (40,15)*+{Y(k_{1}) \times \cdots \times Y(k_{n}) \times \mathbb{G}(k, k_{1} + \cdots + k_{n})} };
(9.5,10)*{\scriptstyle (-\cdot g_{1}, \ldots, -\cdot g_{n}) \times 1};
{\ar (40,15)*+{Y(k_{1}) \times \cdots \times Y(k_{n}) \times \mathbb{G}(k, k_{1} + \cdots + k_{n})}; (80,0)*+{Y^{\star n}(k)} };
{\ar (0,0)*+{Y(k_{1}) \times \cdots \times Y(k_{n}) \times \mathbb{G}(k, k_{1} + \cdots + k_{n})}; (40,-15)*+{Y(k_{1}) \times \cdots \times Y(k_{n}) \times \mathbb{G}(k, k_{1} + \cdots + k_{n})} };
(9.5,-10)*+{\scriptstyle 1 \times \big( (g_{1} + \cdots + g_{n})\cdot - \big)};
{\ar (40,-15)*+{Y(k_{1}) \times \cdots \times Y(k_{n}) \times \mathbb{G}(k, k_{1} + \cdots + k_{n})}; (80,0)*+{Y^{\star n}(k)} };
\endxy
\]
The first map along the top acts using the $g_{i}$'s, while the first map along the bottom is given by
\[
h \mapsto \mu(e_{n}; g_{1}, \ldots, g_{n}) \cdot h
\]
in the final coordinate.

Let $f \in G(n)$, considered as a morphism $n \rightarrow n$ in $\mathbb{G}$.  We induce a map $f \bullet -:Y^{\star n}(k) \rightarrow Y^{\star n}(k)$ using the collection of maps
\[
\prod_{i=1}^{n} Y(k_{i}) \times \mathbb{G}(k, k_{1} + \cdots + k_{n}) \rightarrow \prod_{i=1}^{n} Y(k_{\pi (f)^{-1}(i)}) \times \mathbb{G}(k, k_{1} + \cdots + k_{n})
\]
by using the symmetry $\pi(f)$ on the first $n$ factors and left multiplication by the element $\mu(f; e_{k_{1}}, \ldots, e_{k_{n}})$ on $\mathbb{G}(k, k_{1} + \cdots + k_{n})$.  To induce a map between the coends, we must show that these maps commute with the two lefthand maps in the diagram above.  For the top map, this is merely functoriality of the product together with naturality of the symmetry.  For the bottom map, this is the equation
\[
\mu(f; \overline{e}) \cdot \mu(e; g_{1}, \ldots, g_{n}) = \mu(e; g_{\pi (f)^{-1} 1}, \ldots, g_{\pi (f)^{-1} n}) \cdot \mu(f; \overline{e}).
\]
Both of these are equal to $\mu(f; g_{1}, \ldots, g_{n})$ by the action operad axiom.  Functoriality is then easy to check using that the maps inducing $(f_{1}f_{2}) \bullet -$ are given by the composite of the maps inducing $f_{1} \bullet (f_{2} \bullet -)$.
\end{proof}

We are now ready for the abstract description of the substitution tensor product.  The following proposition is easily checked directly using the definition of the coend; in fact, the righthand side below should be taken as the definition of $X \circ Y$ as both sides are really the result of some colimiting process.

\begin{prop}
Let $X, Y \in \hat{\mathbb{G}}$.  Then
\[
(X \circ Y)(k) \cong \int^{n} X(n) \times Y^{\star n}(k).
\]
\end{prop}

Finally we are in a position to prove Theorem \ref{operad=monoid}.  We make heavy use of the following consequence of the Yoneda lemma:  given any functor $F:\mathbb{G} \rightarrow \mb{Sets}$ and a fixed object $a \in \mathbb{G}$, we have a natural isomorphism
\[
\int^{n \in \mathbb{G}} \mathbb{G}(n,a) \times F(n) \cong F(a);
\]
there is a corresponding result for $F:\mathbb{G}^{\textrm{op}} \rightarrow \mb{Sets}$ using representables of the form $\mathbb{G}(a,n)$ instead.

\begin{proof}[Proof of \ref{operad=monoid}]
First we must show that $\mb{G}\mbox{-}\mb{Coll}$ has a monoidal structure using $\circ$.  To prove this, we must give the unit and associativity isomorphisms and then check the monoidal category axioms.  First, note that the unit object is given as $I = \mathbb{G}(-,1)$.  Then for the left unit isomorphism, we have
\[
\begin{array}{rcl}
I \circ Y (k) & \cong & \int^{n} \mathbb{G}(n,1) \times Y^{\star n}(k) \\
& \cong & Y^{\star 1}(k) \\
& \cong & Y(k)
\end{array}
\]
using only the properties of the coend.  For the right unit isomorphism, we have
\[
\begin{array}{rcl}
X \circ I (k) & \cong & \int^{n} X(n) \times I^{\star n}(k) \\
& \cong & \int^{n} X(n) \times \int^{k_{1}, \ldots, k_{n}} \mathbb{G}(k_{1},1) \times \cdots \times \mathbb{G}(k_{n},1) \times \mathbb{G}(k, k_{1} + \cdots + k_{n}) \\
& \cong & \int^{n} X(n) \times \mathbb{G}(k,1+ \cdots +1) \\
& = & \int^{n} X(n) \times \mathbb{G}(k,n) \\
& \cong & X(k)
\end{array}
\]
using the same methods.

For associativity, we compute $(X \circ Y) \circ Z$ and $X \circ (Y \circ Z)$.
\[
\begin{array}{rcl}
(X \circ Y) \circ Z (k) & = & \int^{m} X \circ Y (m) \times Z^{\star m}(k) \\
& = & \int^{m} \big( \int^{l} X(l) \times Y^{\star l}(m) \big) \times Z^{\star m}(k) \\
& \cong & \int^{m,l} X(l) \times Y^{\star l}(m) \times Z^{\star m}(k) \\
& \cong & \int^{l} X(l) \times \int^{m} Y^{\star l}(m) \times Z^{\star m}(k)
\end{array}
\]
The first isomorphism is from products distributing over colimits and hence coends, and the second is that fact plus the Fubini Theorem for coends \cite{maclane-catwork}.  A similar calculation shows
\[
X \circ (Y \circ Z)(k) \cong \int^{l} X(l) \times (Y \circ Z)^{\star l}(k).
\]
Thus the associativity isomorphism will be induced once we construct an isomorphism $\int^{m} Y^{\star l}(m) \times Z^{\star m} \cong (Y \circ Z)^{\star l}$.  We do this by induction, with the $l=1$ case being the isomorphism $Y^{\star 1} \cong Y$ together with the definition of $\circ.$  Assuming true for $l$, we prove the case for $l+1$ by the calculations below.
\[
\begin{array}{rcl}
(Y \circ Z)^{\star l+1} & \cong & (Y \circ Z) \star (Y \circ Z)^{\star l} \\
& \cong & (Y \circ Z) \star \big( \int^{m} Y^{\star l}(m) \times Z^{\star m} \big) \\
& = & \big( \int^{n} Y(n) \times Z^{\star n} \big) \star \big( \int^{m} Y^{\star l}(m) \times Z^{\star m} \big) \\
& = & \int^{a,b} \big( \int^{n} Y(n) \times Z^{\star n}(a) \big)  \times \big( \int^{m} Y^{\star l}(m) \times Z^{\star m}(b) \big) \times \mathbb{G}(-, a+b) \\
& \cong & \int^{a,b,n,m} Y(n) \times Y^{\star l}(m) \times Z^{\star n}(a) \times Z^{\star m}(b) \times  \mathbb{G}(-, a+b) \\
& \cong & \int^{n,m} Y(n) \times Y^{\star l}(m) \times Z^{\star (n+m)} \\
& \cong & \int^{j} \int^{n,m} Y(n) \times Y^{\star l}(m) \times \mathbb{G}(j, n+m) \times Z^{\star j} \\
& \cong & \int^{j} Y^{\star (l+1)}(j) \times Z^{\star j}
\end{array}
\]
Each isomorphism above arises from the symmetric monoidal structure on $\mb{Sets}$ using products, the monoidal structure on presheaves using $\star$, the properties of the coend, or the fact that products distribute over colimits.

For the monoidal category axioms on $\hat{\mathbb{G}}$, we only need to note that the unit and associativity isomorphisms arise, using the universal properties of the coend, from the unit and associativity isomorphisms on the category of sets together with the interaction between products and colimits.  Hence the monoidal category axioms follow by those same axioms in $\mb{Sets}$ together with the universal property of the coend.

Now we must show that monoids in $(\hat{\mathbb{G}}, \circ)$ are operads.  By the Yoneda lemma, a map of $\mb{G}$-collections $\eta: I \rightarrow X$ corresponds to an element $\textrm{id} \in X(1)$ since $I = \mathbb{G}(-,1)$.  A map $\mu:X \circ X \rightarrow X$ is given by a collection of $G(k)$-equivariant maps $X \circ X (k) \rightarrow X(k)$.  By the universal property of the coend, this is equivalent to giving maps
\[
\mu_{n, \underline{k}}:X(n) \times X(k_{1}) \times \cdots \times X(k_{n}) \times \mathbb{G}(k, k_{1}+\cdots +k_{n}) \rightarrow X(k)
\]
which are compatible with the actions of $G(k)$ (using the hom-set in the source, and the standard right action in the target) as well as each of $G(n), G(k_{1}), \ldots, G(k_{n})$.  The hom-set in $\mathbb{G}$ is nonempty precisely when $k=k_{1} + \cdots + k_{n}$, so we define the operad multiplication $\mu$ for $X$ to be
\[
\mu (x; y_{1}, \ldots, y_{n}) = \mu_{n, \underline{k}}(x; y_{1}, \ldots, y_{n}; e_{k}).
\]
Compatibility with the actions of the  $G(n), G(k_{1}), \ldots, G(k_{n})$ give the equivariance axioms, and the unit and associativity for the monoid structure give the unit and associativity axioms for the operad structure.  Finally, it is easy to check that a map of monoids is nothing more than an operad map which is appropriately equivariant for each $n$.
\end{proof}

\begin{rem}
The above result can be interpreted for $\mb{G}$-operads in an arbitrary cocomplete symmetric monoidal category $\mathcal{V}$ in which tensor distributes over colimits in each variable.  In order to do so, the following changes must be made.  First, cartesian products of objects $X(k)$ must be replaced by the tensor product in $\mathcal{V}$ of the same objects.  Second, any product with a hom-set from $\mathbb{G}$ must be replaced by a copower with the same set (recall that the copower of a set $S$ with an object $X$ is given by the formula $S \odot X = \coprod_{S} X$).  The same changes also allow one to interpret the results below about algebras in such a category, unless noted otherwise.
\end{rem}

An operad is intended to be an abstract description of a certain type of algebraic structure, and the particular instances of that structure are the algebras for that operad.  We begin with the definition of an algebra over a plain operad.

\begin{Defi}\label{opalgax}
Let $O$ be an operad.  An \textit{algebra} for $O$ consists of a set $X$ together with maps $\alpha_{n}:O(n) \times X^{n} \rightarrow X$ such that the following axioms hold.
\begin{enumerate}
\item The element $\textrm{id} \in O(1)$ is a unit in the sense that
\[
\alpha_{1}(\textrm{id}; x) = x
\]
for all $x \in X$.
\item The maps $\alpha_{n}$ are associative in the sense that the diagram
\[
\xy
(0,0)*+{O(n) \times O(k_{1}) \times X^{k_{1}} \times \cdots \times O(k_{n}) \times X^{k_{n}}}="ul";
(75,0)*+{O(n) \times X^{n}}="ur";
(0,-12)*+{O(n) \times O(k_{1}) \times \cdots \times O(k_{n}) \times X^{k_{1}} \times \cdots \times X^{k_{n}}}="ml";
(0,-24)*+{O(\sum k_{i}) \times X^{\sum k_{i}}}="bl";
(75,-24)*+{X}="br";
{\ar^>>>>>>>>>>>>>>{1 \times \alpha_{k_{1}} \times \cdots \alpha_{k_{n}}} "ul"; "ur"};
{\ar^{\alpha_{n}} "ur"; "br"};
{\ar_{\cong} "ul"; "ml"};
{\ar_{\mu \times 1} "ml"; "bl"};
{\ar_{\alpha_{\sum k_{i}}} "bl"; "br"};
\endxy
\]
commutes.
\end{enumerate}
\end{Defi}

Moving on to algebras for a $\mb{G}$-operad, let $P$ be a $\mb{G}$-operad and let $X$ be any set.  Write $P(n) \times_{G(n)} X^{n}$ for the coequalizer of the pair of maps
\[
P(n) \times G(n) \times X^{n} \rightrightarrows P(n) \times X^{n}
\]
of which the first map is the action of $G(n)$ on $P(n)$ and the second map is
\[
P(n) \times G(n) \times X^{n} \rightarrow P(n) \times \Sigma_{n} \times X^{n} \rightarrow P(n) \times X^{n}
\]
using $\pi_{n}:G(n) \rightarrow \Sigma_{n}$ together with the canonical action of $\Sigma_{n}$ on $X^{n}$ by permutation of coordinates: $\sigma \cdot (x_{1}, \ldots, x_{n}) = (x_{\sigma^{-1}(1)}, \ldots, x_{\sigma^{-1}(n)})$.  By the universal property of the coequalizer, a function $f:P(n) \times_{G(n)} X^{n} \rightarrow Y$ can be identified with a function $\tilde{f}:P(n) \times X^{n} \rightarrow Y$ such that
\[
\tilde{f}(p\cdot g; x_{1}, \ldots, x_{n}) = \tilde{f}(p; x_{\pi(g)^{-1}(1)}, \ldots, x_{\pi(g)^{-1}(n)}).
\]

\begin{Defi}
Let $P$ be a $\mb{G}$-operad.  An \textit{algebra} for $P$ consists of a set $X$ together with maps $\alpha_{n}:P(n) \times_{G(n)} X^{n} \rightarrow X$ such that the maps $\tilde{\alpha}_{n}$ satisfy the usual operad algebra axioms given in Definition \ref{opalgax}.
\end{Defi}

\begin{rem}
It is worth noting that the equivariance required for a $P$-algebra is built into the definition above by requiring the existence of the maps $\alpha_{n}$ to be defined on coequalizers, even though the algebra axioms then only use the maps $\tilde{\alpha}_{n}$.  Since every $\mb{G}$-operad has an underlying plain operad (see \ref{pbaop}, applied to the unique map $\mb{T} \rightarrow \mb{G}$), this reflects the fact that the algebras for the $\mb{G}$-equivariant version are always algebras for the plain version, but not conversely.
\end{rem}

\begin{Defi}
The category of algebras for $P$, $P\mbox{-}\mb{Alg}$, has objects the $P$-algebras $(X, \alpha)$ and morphisms $f: (X, \alpha) \rightarrow (Y, \beta)$ those functions $f:X \rightarrow Y$ such that the following diagram commutes for every $n$.
\[
\xy
{\ar^{1 \times f^{n}} (0,0)*+{P(n) \times X^{n}}; (50,0)*+{P(n) \times Y^{n}} };
{\ar^{\tilde{\beta}_{n}} (50,0)*+{P(n) \times Y^{n}}; (50,-15)*+{Y} };
{\ar_{\tilde{\alpha}_{n}} (0,0)*+{P(n) \times X^{n}}; (0,-15)*+{X} };
{\ar_{f} (0,-15)*+{X}; (50,-15)*+{Y} };
\endxy
\]
\end{Defi}

Let $X$ be a set.  Then the endomorphism operad of $X$, denoted $\mathcal{E}_{X}$, is given by the sets $\mathcal{E}_{X}(n) = \mb{Sets}(X^{n}, X)$ with the identity function in $\mathcal{E}_{X}(1)$ giving the unit element and composition of functions giving the composition operation.  Concretely, composition is given by the formula
\[
\mu(f; g_{1}, \ldots, g_{n}) = f \circ (g_{1} \times \cdots \times g_{n}).
\]

\begin{lem}
Let $G$ be an action operad, and let $X$ be a set.  Then $\mathcal{E}_{X}$ carries a canonical $\mb{G}$-operad structure.
\end{lem}
\begin{proof}
$\mathcal{E}_{X}$ is a symmetric operad, so we define the actions by
\[
\mathcal{E}_{X}(n) \times G(n) \stackrel{1 \times \pi_{n}}{\longrightarrow} \mathcal{E}_{X}(n) \times \Sigma_{n} \rightarrow \mathcal{E}_{X}.
\]
\end{proof}

The previous result is really a change-of-structure groups result.  We record the general result as the following proposition and note that the proof is essentially the same as that for the lemma.

\begin{prop}\label{pbaop}
Let $f:\mb{G} \rightarrow \mb{G'}$ be a map of action operads.  Then $f$ induces a functor $f^{*}$ from the category of $\mb{G'}$-operads to the category of $\mb{G}$-operads.
\end{prop}

We can now use endomorphism operads to characterize algebra structures.

\begin{prop}\label{endoalg}
Let $X$ be a set, and $P$ a $\mb{G}$-operad.  Then algebra structures on $X$ are in 1-to-1 correspondence with $\mb{G}$-operad maps $P \rightarrow \mathcal{E}_{X}$.
\end{prop}
\begin{proof}
A map $P(k) \rightarrow \mathcal{E}_{X}(k)$ corresponds, using the closed structure on $\mb{Sets}$, to a map $P(k) \times X^{k} \rightarrow X$.  The monoid homomorphism axioms give the unit and associativity axioms, and the requirement that $P \rightarrow \mathcal{E}_{X}$ be a map of $\mb{G}$-operads gives the equivariance condition.
\end{proof}

\begin{rem}
The proposition above holds for $P$-algebras in any closed symmetric monoidal category.  Having a closed structure (in addition to all small colimits) is a stronger condition than the tensor preserving colimits in each variable, but it is a natural one that arises in many examples.
\end{rem}

\begin{Defi}
Let $P$ be a $\mb{G}$-operad.  Then $P$ induces an endofunctor of $\mb{Sets}$, denoted $\underline{P}$, by the following formula.
 \[
	\underline{P}(X)	 =  \coprod_n P(n) \times_{G(n)} X^n
\]
\end{Defi}

We now have the following proposition; its proof is standard \cite{maygeom}, and we leave it to the reader.

\begin{prop}\label{op=monad1}  Let $P$ be a $\mb{G}$-operad.
\begin{enumerate}
\item The $\mb{G}$-operad structure on $P$ induces a monad structure on $\underbar{P}$.
\item The category of algebras for the operad $P$ is isomorphic to the category of algebras for the monad $\underbar{P}$.
\end{enumerate}
\end{prop} 

\section{Operads in $\mb{Cat}$}

This section will study those $\mb{G}$-operads for which each $P(n)$ is a category, and from here onwards any operad denoted $P$ is in $\mb{Cat}$. The extra structure that this 2-categorical setting provides allows us to consider notions of pseudoalgebras for an operad, as well as pseudomorphisms of operads. The induced monad associated to an operad of this sort can be shown to be a $2$-monad (see \cite{KS} for background on $2$-monads) and we will proceed to show that the notions of pseudoalgebra for both the operad and the associated 2-monad correspond precisely, i.e., there is an isomorphism of $2$-categories between the 2-category with either strict or pseudo-level cells defined operadically and the 2-category with either strict or pseudo-level cells defined 2-monadically.

The associated monad $\underline{P}$ acquires the structure of a $2$-functor as follows. We define $\underline{P}$ on categories much like before as  the coproduct
	\[
		\underline{P}(X) = \coprod_n P(n) \times_{G(n)} X^n,
	\]
whose objects will be written as equivalence classes $[p;x_1,\ldots,x_n]$ where $p \in P(n)$ and each $x_i \in X$, sometimes written as $[p;\underline{x}]$ when there is no confusion. On functors we define $\underline{P}$ in a similar way, exactly as with functions of sets. Given a natural transformation $\alpha \colon f \Rightarrow g$ we define a new natural transformation $\underline{P}(\alpha)$ as follows. The component of $\underline{P}(\alpha)$ at the object
	\[
		[p;x_1,\ldots,x_n]
	\]
is given by the morphism
	\[
		[1_p;\alpha_{x_1},\ldots,\alpha_{x_n}]
	\]
in $\underline{P}(X)$.
It is a simple observation that this constitutes a $2$-functor, and that the components of the unit and multiplication are functors and are $2$-natural.

%
%
First we will set out some conventions and definitions.
\begin{conv}
We will identify maps $\alpha_n \colon P(n) \times_{G(n)} X^n \rightarrow X$ with maps $\tilde{\alpha}_n \colon P(n) \times X^n \rightarrow X$ via the universal property of the coequalizer. Note also that in the following definitions we will often write the composite
    \[
        P(n) \times \prod(P(k_i) \times X^{k_i}) \rightarrow P(n) \times \prod P(k_i) \times X^{\Sigma k_i} \xrightarrow{\mu^P \times 1} P(\Sigma_{k_i}) \times X^{\Sigma k_i}
    \]
simply abbreviated as $\mu^P \times 1$.  Furthermore, instead of using an element $\textrm{id} \in P(1)$ as the operadic unit, we will now denote this as $\eta^{P}:1 \rightarrow P(1)$.
\end{conv}

We begin with the definitions of the pseudo-level cells in the operadic context, and after each specialize to the strict version.

\begin{Defi}
Let $P$ be a $\mb{G}$-operad. A \textit{pseudoalgebra} for $P$ consists of:
    \begin{itemize}
        \item a category $X$,
        \item a family of functors
            \[
                \left(\alpha_n: P(n) \times_{G(n)} X^n \rightarrow X \right)_{n \in \mathbb{N}},
            \]
        \item for each $n, k_1, \ldots, k_n \in \mathbb{N}$, a natural isomorphism
            \[
                \xy
                    (0,0)*+{\scriptstyle P_n \times \prod_{i=1}^n \left(P_{k_i} \times X^{k_i}\right)}="00";
                    (0,-10)*+{\scriptstyle P_n \times \prod_{i=1}^n P_{k_i} \times X^{\Sigma k_i}}="01";
                    (0,-20)*+{\scriptstyle P_{\Sigma k_i} \times X^{\Sigma k_i}}="02";
                    (60,-20)*+{\scriptstyle X}="12";
                    (60,0)*+{\scriptstyle P_n \times X^n}="11";
                    {\ar_{} "00" ; "01"};
                    {\ar^{1 \times \prod \tilde{\alpha}_{k_i}} "00" ; "11"};
                    {\ar^{\tilde{\alpha}_n} "11" ; "12"};
                    {\ar_{\mu^P \times 1} "01" ; "02"};
                    {\ar_>>>>>>>>>>>>>>>>>>>{\tilde{\alpha}_{\Sigma k_i}} "02" ; "12"};
                    {\ar@{=>}^{\phi_{k_1, \ldots, k_n}} (30,-8) ; (30,-12)};
                \endxy
            \]
        \item and a natural isomorphism
            \[
                \xy
                    (0,0)*+{X}="00";
                    (0,-15)*+{1 \times X}="x10";
                    (0,-30)*+{P(1) \times X}="10";
                    (30,-30)*+{X}="11";
                    {\ar_{\eta^P \times 1} "x10" ; "10"};
                    {\ar_{\tilde{\alpha}_1} "10" ; "11"};
                    {\ar^{1} "00" ; "11"};
                    {\ar_{\cong} "00" ; "x10"};
                    {\ar@{=>}^{\phi_\eta} (10,-18) ; (10,-22)};
                \endxy
            \]
    \end{itemize}
satisfying the following axioms.
    \begin{itemize}
        \item For all $n, k_i, m_{ij} \in \mathbb{N}$, the following equality of pasting diagrams holds.
            \[
                \xy
                    (0,0)*+{\scriptstyle P_n \times \prod_i\left(P_{k_i} \times \prod_j \left(P_{m_{ij}} \times X^{m_{ij}}\right)\right)}="00";
                    (60,0)*+{\scriptstyle P_n \times \prod_i \left(P_{k_i} \times X^{k_i}\right)}="10";
                    (0,-30)*+{\scriptstyle P_{\Sigma k_i} \times \prod_i\prod_j\left(P_{m_{ij}} \times X^{m_{ij}}\right)}="02";
                    (30,-50)*+{\scriptstyle P_{\Sigma\Sigma m_{ij}} \times X^{\Sigma \Sigma m_{ij}}}="04";
                    (80,-20)*+{\scriptstyle P_n \times X^n}="12";
                    (80,-50)*+{\scriptstyle X}="14";
                    {\ar^>>>>>>>>>>>>>>{1 \times \prod\left(1 \times \prod \tilde{\alpha}_{m_ij}\right)} "00" ; "10"};
                    {\ar^{1 \times \prod \tilde{\alpha}_{k_i}} "10" ; "12"};
                    {\ar^{\tilde{\alpha}_n} "12" ; "14"};
                    {\ar_{\mu^P \times 1} "00" ; "02"};
                    {\ar_{\mu^P \times 1} "02" ; "04"};
                    {\ar_{\tilde{\alpha}_{\Sigma\Sigma m_{ij}}} "04" ; "14"};
                    (30,-20)*+{\scriptstyle P_n \times \prod_i\left(P_{\Sigma m_{ij}} \times X^{\Sigma m_{ij}}\right)}="22";
                    {\ar^{\mu^P \times 1} "00" ; "22"};
                    {\ar^{1 \times \prod \tilde{\alpha}_{\Sigma m_{ij}}} "22" ; "12"};
                    {\ar^{\mu^P \times 1} "22" ; "04"};
                    (0,-70)*+{\scriptstyle P_n \times \prod_i\left(P_{k_i} \times \prod_j \left(P_{m_{ij}} \times X^{m_{ij}}\right)\right)}="b00";
                    (50,-70)*+{\scriptstyle P_n \times \prod_i \left(P_{k_i} \times X^{k_i}\right)}="b10";
                    (0,-100)*+{\scriptstyle P_{\Sigma k_i} \times \prod_i\prod_j\left(P_{m_{ij}} \times X^{m_{ij}}\right)}="b02";
                    (20,-120)*+{\scriptstyle P_{\Sigma\Sigma m_{ij}} \times X^{\Sigma \Sigma m_{ij}}}="b04";
                    (80,-90)*+{\scriptstyle P_n \times X^n}="b12";
                    (80,-120)*+{\scriptstyle X}="b14";
                    {\ar^>>>>>>>>>{1 \times \prod\left(1 \times \prod \tilde{\alpha}_{m_ij}\right)} "b00" ; "b10"};
                    {\ar^{1 \times \prod \tilde{\alpha}_{k_i}} "b10" ; "b12"};
                    {\ar^{\tilde{\alpha}_n} "b12" ; "b14"};
                    {\ar_{\mu^P \times 1} "b00" ; "b02"};
                    {\ar_{\mu^P \times 1} "b02" ; "b04"};
                    {\ar_{\tilde{\alpha}_{\Sigma\Sigma m_{ij}}} "b04" ; "b14"};
                    (50,-100)*+{\scriptstyle P_{\Sigma k_i} \times X^{\Sigma k_i}}="b22";
                    {\ar_{\mu^P \times 1} "b10" ; "b22"};
                    {\ar^>>>>>>>>>>>>>>>>{1 \times \prod\prod \tilde{\alpha}_{m_{ij}}} "b02" ; "b22"};
                    {\ar^{\tilde{\alpha}_{\Sigma k_i}} "b22" ; "b14"};
                    {\ar@{=>}^{1 \times \prod_i \phi_{m_{i1}, \ldots, m_{ik_{i}}}} (35,-8) ; (35,-12)};
                    {\ar@{=>}^{\phi_{\Sigma m_{1j}, \ldots, \Sigma m_{nj}}} (50,-33) ; (50,-37)};
                    {\ar@{=>}^{\phi_{k_1,\ldots,k_n}} (60,-92) ; (60,-96)};
                    {\ar@{=>}^{\phi_{m_{11}, \ldots, m_{nk_n}}} (30,-108) ; (30,-112)};
                    {\ar@{=} (45,-58) ; (45,-62)};
                \endxy
            \]
        \item Each pasting diagram of the following form is an identity.
            \[
                \xy
                    (0,0)*+{P_n \times X^n}="00";
                    (12,-12)*+{P_n \times (1 \times X)^n}="11";
                    (24,-24)*+{P_n \times (P_1 \times X)^n}="22";
                    (60,-24)*+{P_n \times X^n}="32";
                    (60,-48)*+{X}="34";
                    (24,-36)*+{P_n \times P_1^n \times X^n}="23";
                    (24,-48)*+{P_n \times X^n}="24";
                    {\ar@/^2.5pc/^{1} "00" ; "32"};
                    {\ar^{\tilde{\alpha}_n} "32" ; "34"};
                    {\ar^{\cong} "00" ; "11"};
                    {\ar^>>>{1 \times \left(\eta^P \times 1\right)^n} "11" ; "22"};
                    {\ar^>>>>>>{1 \times \tilde{\alpha}_1^n} "22" ; "32"};
                    {\ar@/_3pc/_{1} "00" ; "24"};
                    {\ar_{\cong} "22" ; "23"};
                    {\ar_{\mu^P \times 1} "23" ; "24"};
                    {\ar_{\tilde{\alpha}_n} "24" ; "34"};
                    {\ar@{=>}^{1 \times \phi_\eta^n} (32,-8) ; (32,-12)};
                    {\ar@{=>}^{\phi_{1,\ldots,1}} (40,-34) ; (40,-38)};
                \endxy
            \]
    \end{itemize}
\end{Defi}

\begin{Defi}
Let $P$ be a $\mb{G}$-operad. A \textit{ strict algebra} for $P$ consists of a pseudoalgebra in which all of the isomorphisms $\phi$ are identities.
\end{Defi}

\begin{Defi}
Let $(X, \alpha_n,\phi,\phi_\eta)$ and $(Y, \beta_n,\psi,\psi_{\eta})$ be pseudoalgebras for a $\mb{G}$-operad $P$. A \textit{pseudomorphism} of $P$-pseudoalgebras consists of:
    \begin{itemize}
        \item a functor $f \colon X \rightarrow Y$
        \item and a family of natural isomorphisms
            \[
                \xy
                    (0,0)*+{P_n \times X^n}="00";
                    (20,0)*+{X}="10";
                    (0,-15)*+{P_n \times Y^n}="01";
                    (20,-15)*+{Y}="11";
                    {\ar^>>>>>{\tilde{\alpha}_n} "00" ; "10"};
                    {\ar^{f} "10" ; "11"};
                    {\ar_{1 \times f^n} "00" ; "01"};
                    {\ar_>>>>>{\tilde{\beta}_n} "01" ; "11"};
                    {\ar@{=>}^{\overline{f}_n} (10,-5.5) ; (10,-9.5)};
                \endxy
            \]
        \end{itemize}
satisfying the following axioms.
    \begin{itemize}
        \item The following equality of pasting diagrams holds.
            \[
                \xy
                    (0,0)*+{\scriptstyle P_n \times \prod_i (P_{k_i} \times X^{k_i})}="00";
                    (50,0)*+{\scriptstyle P_n \times \prod_i (P_{k_i} \times Y^{k_i})}="10";
                    (0,-25)*+{\scriptstyle P_{\Sigma k_i} \times X^{\Sigma k_i}}="01";
                    (50,-25)*+{\scriptstyle P_{\Sigma k_i} \times Y^{\Sigma k_i}}="11";
                    (75,-15)*{\scriptstyle P_n \times Y^n}="21";
                    (75,-40)*+{\scriptstyle Y}="22";
                    (25,-40)*+{\scriptstyle X}="02";
                    {\ar^{1 \times \prod(1 \times f^{k_i})} "00" ; "10"};
                    {\ar^{1 \times \prod \tilde{\beta}_{k_i}} "10" ; "21"};
                    {\ar_{\mu^P \times 1} "00" ; "01"};
                    {\ar_{\tilde{\alpha}_{\Sigma k_i}} "01" ; "02"};
                    {\ar_{f} "02" ; "22"};
                    {\ar^{1 \times f^{\Sigma k_i}} "01" ; "11"};
                    {\ar_{\tilde{\beta}_{\Sigma k_i}} "11" ; "22"};
                    {\ar_{\mu^P \times 1} "10" ; "11"};
                    {\ar^{\tilde{\beta}_n} "21" ; "22"};
                    {\ar@{=>}^{\overline{f}_n} (37.5,-30.5) ; (37.5,-34.5)};
                    {\ar@{=>}^{\psi_{k_1,\ldots,k_n}} (57.5,-16.5) ; (57.5,-20.5)};
                    (0,-55)*+{\scriptstyle P_n \times \prod_i (P_{k_i} \times X^{k_i})}="b00";
                    (50,-55)*+{\scriptstyle P_n \times \prod_i (P_{k_i} \times Y^{k_i})}="b10";
                    (0,-80)*+{\scriptstyle P_{\Sigma k_i} \times X^{\Sigma k_i}}="b01";
                    (25,-70)*+{\scriptstyle P_n \times X^n}="b11";
                    (75,-70)*{\scriptstyle P_n \times Y^n}="b21";
                    (75,-95)*+{\scriptstyle Y}="b22";
                    (25,-95)*+{\scriptstyle X}="b02";
                    {\ar^{1 \times \prod(1 \times f^{k_i})} "b00" ; "b10"};
                    {\ar^{1 \times \prod \tilde{\beta}_{k_i}} "b10" ; "b21"};
                    {\ar_{\mu^P \times 1} "b00" ; "b01"};
                    {\ar_{\tilde{\alpha}_{\Sigma k_i}} "b01" ; "b02"};
                    {\ar_{f} "b02" ; "b22"};
                    {\ar^{\tilde{\beta}_n} "b21" ; "b22"};
                    {\ar^{1 \times \prod \tilde{\alpha}_{k_i}} "b00" ; "b11"};
                    {\ar^{1 \times f^n} "b11" ; "b21"};
                    {\ar_{\tilde{\alpha}_n} "b11" ; "b02"};
                    {\ar@{=>}^{\overline{f}_n} (50,-80.5) ; (50,-84.5)};
                    {\ar@{=>}^{1 \times \prod\overline{f}_{k_i}} (37.5,-60.5) ; (37.5,-64.5)};
                    {\ar@{=>}^{\phi_{k_1,\ldots,k_n}} (9,-72) ; (9,-76)};
                    {\ar@{=} (37.5,-45.5) ; (37.5,-49.5)};
                \endxy
            \]
            \item The following equality of pasting diagrams holds.
                \[
                    \xy
                        (0,0)*+{X}="00";
                        (20,0)*+{Y}="10";
                        (0,-15)*+{1 \times X}="01";
                        (20,-15)*+{1 \times Y}="11";
                        (0,-30)*+{P_1 \times X}="02";
                        (20,-30)*+{P_1 \times Y} = "12";
                        (20,-45)*+{X}="r02";
                        (40,-45)*+{Y}="r12";
                        {\ar^{f} "00" ; "10"};
                        {\ar@/^2pc/^{1} "10" ; "r12"};
                        {\ar_{\cong} "00" ; "01"};
                        {\ar_{\eta^P \times 1} "01" ; "02"};
                        {\ar_{\tilde{\alpha}_1} "02" ; "r02"};
                        {\ar^{1 \times f} "01" ; "11"};
                        {\ar^{1 \times f} "02" ; "12"};
                        {\ar^{\tilde{\beta}_1} "12" ; "r12"};
                        {\ar_{\cong} "10" ; "11"};
                        {\ar_{\eta^P \times 1} "11" ; "12"};
                        {\ar_{f} "r02" ; "r12"};
                        {\ar@{=>}^{\overline{f}_1} (20,-35.5) ; (20,-39.5)};
                        {\ar@{=>}^{\psi_{\eta}} (30,-20) ; (30,-24)};
                        (60,0)*+{X}="x00";
                        (80,0)*+{Y}="x10";
                        (60,-15)*+{1 \times X}="x01";
                        (60,-30)*+{P_1 \times X}="x02";
                        (80,-45)*+{X}="xr02";
                        (100,-45)*+{Y}="xr12";
                        {\ar^{f} "x00" ; "x10"};
                        {\ar@/^2pc/^{1} "x10" ; "xr12"};
                        {\ar_{\cong} "x00" ; "x01"};
                        {\ar_{\eta^P \times 1} "x01" ; "x02"};
                        {\ar_{\tilde{\alpha}_1} "x02" ; "xr02"};
                        {\ar_{f} "xr02" ; "xr12"};
                        {\ar@/^2pc/^{1} "x00" ; "xr02"};
                        {\ar@{=>}^{\phi_\eta} (70,-20) ; (70,-24)};
                        {\ar@{=} (45,-22.5) ; (49,-22.5)};
                    \endxy
                \]
    \end{itemize}
\end{Defi}

\begin{Defi}
Let $(X, \alpha_n,\phi,\phi_\eta)$ and $(Y, \beta_n,\psi,\psi_{\eta})$ be pseudoalgebras for a $\mb{G}$-operad $P$. A \textit{strict morphism} of $P$-pseudoalgebras consists of a pseudomorphism in which all of the isomorphisms $\overline{f}_{n}$ are identities.
\end{Defi}

\begin{rem}
A strict algebra for a $\mb{G}$-operad $P$ in $\mb{Cat}$ is precisely the same thing as an algebra for $P$ considered as an operad on the \textit{category} of small categories and functors.  A strict morphism between strict algebras is then just a map of $P$-algebras in the standard sense.  We could also consider the notion of a lax algebra for an operad, or a lax morphism of algebras, simply by considering natural transformations in place of isomorphisms in the definitions.
\end{rem}

\begin{Defi}
Let $P$ be a $\mb{G}$-operad and let $f, g \colon (X, \alpha, \phi, \phi_\eta) \rightarrow (Y, \beta, \psi, \psi_\eta)$ be pseudomorphisms of $P$-pseudoalgebras. A \textit{$P$-transformation} is then a natural transformation $\gamma \colon f \Rightarrow g$ such that the following following equality of pasting diagrams holds, for all $n$.
    \[
        \xy
            (0,0)*+{P_n \times X^n}="00";
            (30,0)*+{P_n \times Y^n}="10";
            (0,-20)*+{X}="01";
            (30,-20)*+{Y}="11";
            {\ar@/^1.5pc/^{1 \times f^n} "00" ; "10"};
            {\ar_{1 \times g^n} "00" ; "10"};
            {\ar^{\tilde{\beta}_n} "10" ; "11"};
            {\ar_{\tilde{\alpha}_n} "00" ; "01"};
            {\ar_{g} "01" ; "11"};
            {\ar@{=>}^{1 \times \gamma^n} (13.5,5.5) ; (13.5,1.5)};
            {\ar@{=>}^{\overline{g}_n} (13.5,-8) ; (13.5,-12)};
            (60,0)*+{P_n \times X^n}="x00";
            (90,0)*+{P_n \times Y^n}="x10";
            (60,-20)*+{X}="x01";
            (90,-20)*+{Y}="x11";
            {\ar^{1 \times f^n} "x00" ; "x10"};
            {\ar^{\tilde{\beta}_n} "x10" ; "x11"};
            {\ar_{\tilde{\alpha}_n} "x00" ; "x01"};
            {\ar^{f} "x01" ; "x11"};
            {\ar@/_1.5pc/_{g} "x01" ; "x11"};
            {\ar@{=>}^{\gamma} (75,-21.5) ; (75,-25.5)};
            {\ar@{=>}^{\overline{f}_n} (75,-8) ; (75,-12)};
            {\ar@{=} (42.75,-10) ; (46.75,-10)};
        \endxy
    \]
\end{Defi}

We can form various 2-categories using these cells.

\begin{Defi}
Let $P$ be a $\mb{G}$-operad.
\begin{itemize}
\item The $2$-category $P\mbox{-}\mb{Alg}_{s}$ consists of strict $P$-algebras, strict morphisms, and $P$-transformations.
\item The $2$-category $\mb{Ps}\mbox{-}P\mbox{-}\mb{Alg}$ consists of $P$-pseudoalgebras, pseudomorphisms, and $P$-transformations.
\end{itemize}
\end{Defi}

We also have the corresponding 2-monadic definitions, which we give for completeness.  We state these for any 2-category $\m{K}$, as specializing to $\mb{Cat}$ does not simplify them in any way.

\begin{Defi}
Let $T \colon \m{K} \rightarrow \m{K}$ be a $2$-monad. A $T$-\textit{pseudoalgebra} consists of an object $X$, a $1$-cell $\alpha \colon TX \rightarrow X$, and invertible $2$-cells
    \[
        \xy
            (0,0)*+{T^2X}="00";
            (20,0)*+{TX}="10";
            (0,-15)*+{TX}="01";
            (20,-15)*+{X}="11";
            {\ar^{T\alpha} "00" ; "10"};
            {\ar^{\alpha} "10" ; "11"};
            {\ar_{\mu_X} "00" ;  "01"};
            {\ar_{\alpha} "01" ; "11"};
            {\ar@{=>}^{\Phi} (10,-5.5) ; (10,-9.5)};
            (40,0)*+{X}="20";
            (52.5,-15)*+{TX}="31";
            (72.5,-15)*+{X}="41";
            {\ar_{\eta_X} "20" ; "31"};
            {\ar_{\alpha} "31" ; "41"};
            {\ar@/^1.5pc/^{1_X} "20" ; "41"};
            {\ar@{=>}^{\Phi_{\eta}} (54.5,-5.5) ; (54.5,-9.5)};
        \endxy
    \]
satisfying the following axioms.
    \begin{itemize}
        \item The following equality of pasting diagrams holds.
    \[
        \xy
            (5,0)*+{T^3X}="t3xl";
            (29,0)*+{T^2X}="t2xl1";
            (5,-17.5)*+{T^2X}="t2xl2";
            (24,-35)*+{TX}="txl1";
            (48,-17.5)*+{TX}="txl2";
            (48,-35)*+{X}="xl";
            (24,-17.5)*+{T^2X}="t2xl3";
            {\ar^{T^2\alpha} "t3xl" ; "t2xl1"};
            {\ar^{T\alpha} "t2xl1" ; "txl2"};
            {\ar^{\alpha} "txl2" ; "xl"};
            {\ar_{\mu_{TX}} "t3xl" ; "t2xl2"};
            {\ar_{\mu_X} "t2xl2" ; "txl1"};
            {\ar_{\alpha} "txl1" ; "xl"};
            {\ar_{T\mu_X} "t3xl" ; "t2xl3"};
            {\ar^{T\alpha} "t2xl3" ; "txl2"};
            {\ar_{\mu_X} "t2xl3" ; "txl1"};
            {\ar@{=>}_{T\Phi} (26,-6) ; (26,-10)};
            {\ar@{=>}^{\Phi} (36,-24) ; (36,-28)};
            (64,0)*+{T^3X}="t3xr";
            (88,0)*+{T^2X}="t2xr1";
            (64,-17.5)*+{T^2X}="t2xr2";
            (83,-35)*+{TX}="txr1";
            (107,-17.5)*+{TX}="txr2";
            (107,-35)*+{X}="xr";
            (88,-17.5)*+{TX}="txr3";
            {\ar^{T^2\alpha} "t3xr" ; "t2xr1"};
            {\ar^{T\alpha} "t2xr1" ; "txr2"};
            {\ar^{\alpha} "txr2" ; "xr"};
            {\ar_{\mu_{TX}} "t3xr" ; "t2xr2"};
            {\ar_{\mu_X} "t2xr2" ; "txr1"};
            {\ar_{\alpha} "txr1" ; "xr"};
            {\ar_{T\alpha} "t2xr2" ; "txr3"};
            {\ar_{\alpha} "txr3" ; "xr"};
            {\ar_{\mu_X} "t2xr1" ; "txr3"};
            {\ar@{=>}_{\Phi} (98,-15) ; (98,-19)};
            {\ar@{=>}^{\Phi} (85,-24) ; (85,-28)};
            {\ar@{=} (54,-20) ; (56,-20)};
        \endxy
    \]
    \item The following pasting diagram is an identity.
    \[
        \xy
            (0,0)*+{TX}="txl1";
            (15,-15)*+{T^2X}="t2x";
            (15,-30)*+{TX}="txl2";
            (35,-15)*+{TX}="txl3";
            (35,-30)*+{X}="xl";
            {\ar@/^1.7pc/^{1_{TX}} "txl1" ; "txl3"};
            {\ar@/_1.7pc/_{1_{TX}} "txl1" ; "txl2"};
            {\ar_{T\eta_X} "txl1" ; "t2x"};
            {\ar^{T\alpha} "t2x" ; "txl3"};
            {\ar_{\mu_X} "t2x" ; "txl2"};
            {\ar_{\alpha} "txl2" ; "xl"};
            {\ar^{\alpha} "txl3" ; "xl"};
            {\ar@{=>}^{T\Phi_\eta} (17,-5.5) ; (17,-9.5)};
            {\ar@{=>}^{\Phi} (25,-20.5) ; (25,-24.5)};
        \endxy
    \]
    \end{itemize}
\end{Defi}

\begin{Defi}
Let $T \colon \m{K} \rightarrow \m{K}$ be a $2$-monad. A \textit{strict $T$-algebra} consists of a pseudoalgebra in which all of the isomorphisms $\Phi$ are identities.
\end{Defi}

\begin{Defi}
Let $T$ be a $2$-monad and let $(X,\alpha,\Phi,\Phi_\eta)$, $(Y,\beta,\Psi,\Psi_\eta)$ be $T$-pseudoalgebras. A \textit{pseudomorphism} $(f, \bar{f})$ between these pseudoalgebras consists of a $1$-cell $f \colon X \rightarrow Y$ along with an invertible $2$-cell
    \[
        \xy
            (0,0)*+{TX}="00";
            (20,0)*+{TY}="10";
            (0,-15)*+{X}="01";
            (20,-15)*+{Y}="11";
            {\ar^{Tf} "00" ; "10"};
            {\ar^{\beta} "10" ; "11"};
            {\ar_{\alpha} "00" ; "01"};
            {\ar_{f} "01" ; "11"};
            {\ar@{=>}^{\bar{f}} (10,-5.5) ; (10,-9.5)};
        \endxy
    \]
satisfying the following axioms.
    \begin{itemize}
        \item The following equality of pasting diagrams holds.
                \[
        \xy
            (5,0)*+{T^2X}="t3xl";
            (29,0)*+{T^2Y}="t2xl1";
            (5,-17.5)*+{TX}="t2xl2";
            (24,-35)*+{TX}="txl1";
            (48,-17.5)*+{TY}="txl2";
            (48,-35)*+{Y}="xl";
            (24,-17.5)*+{TX}="t2xl3";
            {\ar^{T^2f} "t3xl" ; "t2xl1"};
            {\ar^{T\beta} "t2xl1" ; "txl2"};
            {\ar^{\beta} "txl2" ; "xl"};
            {\ar_{\mu_X} "t3xl" ; "t2xl2"};
            {\ar_{\alpha} "t2xl2" ; "txl1"};
            {\ar_{f} "txl1" ; "xl"};
            {\ar^{T\alpha} "t3xl" ; "t2xl3"};
            {\ar^{Tf} "t2xl3" ; "txl2"};
            {\ar_{\alpha} "t2xl3" ; "txl1"};
            {\ar@{=>}^{T\bar{f}} (24,-6) ; (24,-10)};
            {\ar@{=>}^{\bar{f}} (36,-24) ; (36,-28)};
            {\ar@{=>}^{\Phi} (13.5,-15.5) ; (13.5,-19.5)};
            (64,0)*+{T^2X}="t3xr";
            (88,0)*+{T^2Y}="t2xr1";
            (64,-17.5)*+{TX}="t2xr2";
            (83,-35)*+{TX}="txr1";
            (107,-17.5)*+{TY}="txr2";
            (107,-35)*+{Y}="xr";
            (88,-17.5)*+{TX}="txr3";
            {\ar^{T^2f} "t3xr" ; "t2xr1"};
            {\ar^{T\beta} "t2xr1" ; "txr2"};
            {\ar^{\beta} "txr2" ; "xr"};
            {\ar_{\mu_{X}} "t3xr" ; "t2xr2"};
            {\ar_{\alpha} "t2xr2" ; "txr1"};
            {\ar_{f} "txr1" ; "xr"};
            {\ar_{Tf} "t2xr2" ; "txr3"};
            {\ar_{\beta} "txr3" ; "xr"};
            {\ar_{\mu_Y} "t2xr1" ; "txr3"};
            {\ar@{=>}_{\Psi} (98,-15) ; (98,-19)};
            {\ar@{=>}^{\bar{f}} (85,-24) ; (85,-28)};
            {\ar@{=} (54,-20) ; (56,-20)};
        \endxy
    \]
    \item The following equality of pasting diagrams holds.
            \[
                        \xy
            (0,0)*+{X}="00";
            (20,0)*+{Y}="10";
            (0,-20)*+{TX}="01";
            (20,-20)*+{TY}="11";
            (10,-35)*+{X}="02";
            (30,-35)*+{Y}="12";
            {\ar^{f} "00" ; "10"};
            {\ar@/^1.5pc/^{1_Y} "10" ; "12"};
            {\ar_{\eta_X} "00" ; "01"};
            {\ar_{\eta_Y} "10" ; "11"};
            {\ar_{Tf} "01" ; "11"};
            {\ar_{\alpha} "01" ; "02"};
            {\ar_{f} "02" ; "12"};
            {\ar^{\beta} "11" ; "12"};
            {\ar@{=>}^{\bar{f}} (15,-25.5) ; (15,-29.5)};
            {\ar@{=>}^{\Psi_{\eta}} (25,-17) ; (25,-21)};
            (50,0)*+{X}="30";
            (70,0)*+{Y}="40";
            (50,-20)*+{TX}="31";
            (60,-35)*+{X}="32";
            (80,-35)*+{Y}="42";
            {\ar^{f} "30" ; "40"};
            {\ar_{\eta_X} "30" ; "31"};
            {\ar_{\alpha} "31" ; "32"};
            {\ar_{f} "32" ; "42"};
            {\ar@/^1.5pc/^{1_X} "30" ; "32"};
            {\ar@/^1.5pc/^{1_Y} "40" ; "42"};
            {\ar@{=>}^{\Phi_{\eta}} (55,-17) ; (55,-21)};
        \endxy
        \]
\end{itemize}
\end{Defi}

\begin{Defi}
Let $T$ be a $2$-monad and let $(X,\alpha,\Phi,\Phi_\eta)$, $(Y,\beta,\Psi,\Psi_\eta)$ be $T$-pseudoalgebras. A \textit{strict morphism} $(f, \bar{f})$ consists of a pseudomorphism in which $\bar{f}$ is an identity.
\end{Defi}

\begin{rem}
Once again, the strict algebras and strict morphisms are exactly the same as algebras and morphisms for the underlying monad on the underlying category of $\m{K}$.
\end{rem}

\begin{Defi}
Let $(f, \overline{f}), (g, \overline{g}):X \rightarrow Y$ be pseudomorphisms of $T$-algebras.  A \textit{$T$-transformation} consists of a 2-cell $\gamma:f \Rightarrow g$ such that the following equality of pasting diagrams holds.
\[
        \xy
            (0,0)*+{TX}="00";
            (30,0)*+{TY}="10";
            (0,-20)*+{X}="01";
            (30,-20)*+{Y}="11";
            {\ar@/^1.5pc/^{Tf} "00" ; "10"};
            {\ar_{Tg} "00" ; "10"};
            {\ar^{\beta} "10" ; "11"};
            {\ar_{\alpha} "00" ; "01"};
            {\ar_{g} "01" ; "11"};
            {\ar@{=>}^{T \gamma} (13.5,5.5) ; (13.5,1.5)};
            {\ar@{=>}^{\overline{g}} (13.5,-8) ; (13.5,-12)};
            (60,0)*+{TX}="x00";
            (90,0)*+{TY}="x10";
            (60,-20)*+{X}="x01";
            (90,-20)*+{Y}="x11";
            {\ar^{Tf} "x00" ; "x10"};
            {\ar^{\beta} "x10" ; "x11"};
            {\ar_{\alpha} "x00" ; "x01"};
            {\ar^{f} "x01" ; "x11"};
            {\ar@/_1.5pc/_{g} "x01" ; "x11"};
            {\ar@{=>}^{\gamma} (75,-21.5) ; (75,-25.5)};
            {\ar@{=>}^{\overline{f}} (75,-8) ; (75,-12)};
            {\ar@{=} (42.75,-10) ; (46.75,-10)};
        \endxy
    \]
\end{Defi}

Once again, we have 2-categories defined using the different kinds of cells.

\begin{Defi}
Let $T$ be a 2-monad.
\begin{itemize}
\item The $2$-category $T\mbox{-}\mb{Alg}_{s}$ consists of strict $T$-algebras, strict morphisms, and $T$-transformations.
\item The $2$-category $\mb{Ps}\mbox{-}T\mbox{-}\mb{Alg}$ consists of $T$-pseudoalgebras, pseudomorphisms, and $T$-transformations.
\end{itemize}
\end{Defi}

Our main result in this section is the following, showing that one can consider algebras and higher cells, in either strict or pseudo strength, using either the operadic or 2-monadic incarnation of a $\mb{G}$-operad $P$.  This extends Proposition \ref{op=monad1}.

\begin{thm}
Let $P$ be a $\mb{G}$-operad in $\mb{Cat}$.
\begin{itemize}
\item There is an isomorphism of $2$-categories
    \[
        P\mbox{-}\mb{Alg}_{s} \cong \underline{P}\mbox{-}\mb{Alg}_{s}.
    \]
\item There is an isomorphism of $2$-categories
    \[
        \mb{Ps}\mbox{-}P\mbox{-}\mb{Alg} \cong \mb{Ps}\mbox{-}\underline{P}\mbox{-}\mb{Alg}
    \]
    extending the one above.
\end{itemize}
\end{thm}
\begin{proof}
A proof of the first statement follows from our proof of the second by inserting identities where appropriate.  Thus we begin by constructing a $2$-functor $R \colon \mb{Ps}\mbox{-}\underline{P}\mbox{-}\mb{Alg} \rightarrow \mb{Ps}\mbox{-}P\mbox{-}\mb{Alg}$. We map a $\underline{P}$-pseudoalgebra $(X,\alpha,\Phi,\Phi_\eta)$ to the following $P$-pseudoalgebra on the same object $X$. First we define the functor $\alpha_n$ to be the composite
    \[
        \xy
            (0,0)*+{\alpha_n \colon P(n) \times_{G(n)} X^n}="00";
            (35,0)*+{\underline{P}(X)}="10";
            (55,0)*+{X.}="20";
            {\ar@{^{(}->} "00" ; "10"};
            {\ar^{\alpha} "10" ; "20"};
        \endxy
    \]
The isomorphisms $\phi_{k_1,\ldots,k_n}$ are defined using $\Phi$ as in the following diagram
	\[
		\xy
			(-10,0)*+{\scriptstyle P_n \times \prod_{i=1}^n\left(P_{k_i} \times X^{k_i}\right)}="00";
           	(30,0)*+{\scriptstyle P_n \times \prod_i \left( P_{k_i} \times_{G_{k_i}} X^{k_i} \right)}="10";
            (60,0)*+{\scriptstyle P_n \times \underline{P}(X)^n}="20";
            (90,0)*+{\scriptstyle P_n \times X^n}="30";
            (-10,-15)*+{\scriptstyle P_n \times \prod_{i} P_{k_i} \times X^{\Sigma k_I}}="01";
            (-10,-30)*+{\scriptstyle P_{\Sigma k_i} \times X^{\Sigma k_{i}}}="02";
            (60,-7.5)*+{\scriptstyle P_n \times_{G_n} \underline{P}(X)^n}="21";
            (60,-15)*+{\scriptstyle \underline{P}^2(X)}="22";
            (90,-7.5)*+{\scriptstyle P_n \times_{G_n} X^n}="31";
            (90,-15)*+{\scriptstyle \underline{P}(X)}="32";
            (30,-30)*+{\scriptstyle P_{\Sigma k_i} \times_{G_{\Sigma k_i}} X^{\Sigma k_i}}="12";
            (60,-30)*+{\scriptstyle \underline{P}(X)}="23";
            (90,-30)*+{\scriptstyle X}="33";
            {\ar "00" ; "10"};
            {\ar "00" ; "01"};
            {\ar_{\mu^P \times 1} "01" ; "02"};
            {\ar@{^{(}->} "10" ; "20"};
            {\ar "20" ; "21"};
            {\ar^{1 \times \alpha^n} "20" ; "30"};
            {\ar "30" ; "31"};
            {\ar@{^{(}->} "21" ; "22"};
            {\ar^{\underline{P}\alpha} "22" ; "32"};
            {\ar@{^{(}->} "31" ; "32"};
            {\ar_{\mu_X} "22" ; "23"};
            {\ar_{\alpha} "23" ; "33"};
            {\ar^{\alpha} "32" ; "33"};
            {\ar "02" ; "12"};
            {\ar@{^{(}->} "12" ; "23"};
            {\ar@{=>}^{\Phi} (75,-20.5) ; (75,-24.5)};
        \endxy
    \]
whilst $\Phi_\eta$ is simply sent to itself, since the composition of $\alpha$ with the composite of the coequalizer and inclusion map from $P(1) \times X$ into $\underline{P}(X)$ is just $\tilde{\alpha_1}$. Checking the axioms here is most easily done on components and it is easily seen that the axioms required of this data to be a $P$-pseudoalgebra are precisely those that they satisfy by virtue of $X$ being  a $\underline{P}$-pseudoalgebra.

For a $1$-cell $(f,\overline{f}) \colon (X, \alpha) \rightarrow (Y, \beta)$, we send $f$ to itself whilst sending $\overline{f}$ to the obvious family of isomorphisms, as follows.
    \[
        \xy
            (-30,0)*+{P(n) \times X^n}="-10";
            (-30,-15)*+{P(n) \times Y^n}="-11";
            (0,0)*+{P(n) \times_{G(n)} X^n}="00";
            (30,0)*+{\underline{P}(X)}="10";
            (60,0)*+{X}="20";
            (0,-15)*+{P(n) \times_{G(n)} Y^n}="01";
            (30,-15)*+{\underline{P}(Y)}="11";
            (60,-15)*+{Y}="21";
            {\ar@{^{(}->} "00" ; "10"};
            {\ar^{\alpha} "10" ; "20"};
            {\ar_{1 \times f^n} "00" ; "01"};
            {\ar_{\underline{P}f} "10" ; "11"};
            {\ar^{f} "20" ; "21"};
            {\ar@{^{(}->} "01" ; "11"};
            {\ar_{\beta} "11" ; "21"};
            {\ar "-10" ; "00"};
            {\ar "-11" ; "01"};
            {\ar_{1 \times f^n} "-10" ; "-11"};
            {\ar@{=>}^{\overline{f}} (45,-5.5) ; (45,-9.5)};
        \endxy
    \]
It is easy to check that the above data satisfy the axioms for being a pseudomorphism of $P$-pseudoalgebras, following from the axioms for $(f,\overline{f})$ being a pseudomorphism of $\underline{P}$-pseudoalgebras. A $\underline{P}$-transformation $\gamma \colon (f, \bar{f}) \Rightarrow (g, \bar{g})$ immediately gives a $P$-transformation $\bar{\gamma}$ between the families of isomorphisms we previously defined, with the components of $\bar{\gamma}$ being precisely those of $\gamma$.  It is then obvious that $R$ is a $2$-functor.

For there to be an isomorphism of $2$-categories, we require an inverse to $R$, namely a $2$-functor $S \colon \mb{Ps}\mbox{-}P\mbox{-}\mb{Alg} \rightarrow \mb{Ps}\mbox{-}\underline{P}\mbox{-}\mb{Alg}$. Now assume that $(X, \alpha_n, \phi_{\underline{k}_i}, \phi_\eta)$ is a $P$-pseudoalgebra.  We will give the same object $X$ a $\underline{P}$-pseudoalgebra structure. We can induce a functor $\alpha \colon \underline{P}(X) \rightarrow X$ by using the universal property of the coproduct.
    \[
        \xy
            (-30,0)*+{P(n) \times X^n}="-10";
            (0,0)*+{P(n) \times_{G(n)} X^n}="00";
            (30,0)*+{\underline{P}(X)}="10";
            (30,-15)*+{X}="11";
            {\ar "-10" ; "00"};
            {\ar^{\alpha_n} "00" ; "11"};
            {\ar@{^{(}->} "00" ; "10"};
            {\ar^{\exists ! \alpha} "10" ; "11"};
            {\ar_{\tilde{\alpha}_n} "-10" ; "11"};
        \endxy
    \]
Of course, this can be induced using either $\alpha_n$ or $\tilde{\alpha}_n$, each giving the same functor $\alpha$ by uniqueness. The components of the isomorphism $\Phi \colon \alpha \circ \underline{P}(\alpha) \Rightarrow \alpha \circ \mu_X$ can be given as follows. Let $|\underline{x}_i|$ denote the number of objects in the list $\underline{x}_i$. Then define the component of $\Phi$ at the object
    \[
        [p;[q_1;\underline{x}_1],\ldots,[q_n;\underline{x}_n]]
    \]
to be component of $\phi_{|\underline{x}_1|, \ldots, |\underline{x}_n|}$ at the same object. To make this clearer, consider the object $[p;[q_1;x_{11}],[q_2;x_{21},x_{22}],[q_3;x_{31}]]$. The component of $\Phi$ at this object is given by the component of $\phi_{1,2,1}$ at the same object. The isomorphism $\phi_\eta$ is again sent to itself.

Now given a $1$-cell $f$ with structure $2$-cells $\overline{f}_n$ we define a $1$-cell $(F,\overline{F})$ with underlying $1$-cell $f$ and structure $2$-cell $\overline{F}$ with components
    \[
        \overline{F}_{[p;x_1, \ldots, x_n]} := \left(\overline{f}_{n}\right)_{(p;x_1,\ldots,x_n)}.
    \]
For example, the component of $\overline{F}$ at the object $[p;x_1,x_2,x_3]$ would be the component of $f_3$ at the object $(p;x_1,x_2,x_3)$.

The mapping for $2$-cells is just the identity as before. These mappings again constitute a $2$-functor in the obvious way and from how they are defined it is also clear that this is an inverse to $R$.
\end{proof}

\begin{rem}
Another interpretation of pseudoalgebras can be given in terms of pseudomorphisms of operads. Algebras for an operad $P$ can be identified with a morphism of operads $F \colon P \rightarrow \mathcal{E}_X$, where $\mathcal{E}_X$ is the endomorphism operad (Proposition \ref{endoalg}). We can similarly define pseudomorphisms for a $\mathbf{Cat}$-enriched $\mb{G}$-operad and identify pseudoalgebras with pseudomorphisms into the endomorphism operad.

If $P$, $Q$ are $\mb{G}$-operads then a \textit{pseudomorphism} of $\mb{G}$-operads $F \colon P \rightarrow Q$ consists of a family of $\mb{G}$-equivariant functors
            \[
                \left(F_n \colon P(n) \rightarrow Q(n)\right)_{n \in \mathbb{N}}
            \]
together with isomorphisms instead of the standard algebra axioms.  For example, the associativity isomorphism has the following form.
            \[
                \xy
                    (0,0)*+{\scriptstyle P(n) \times \prod_i P(k_i)}="00";
                    (35,0)*+{\scriptstyle Q(n) \times \prod_i Q(k_i)}="10";
                    (0,-15)*+{\scriptstyle P(\Sigma k_i)}="01";
                    (35,-15)*+{\scriptstyle Q(\Sigma k_i)}="11";
                    {\ar^{F_n \times \prod_i F_{k_i}} "00" ; "10"};
                    {\ar^{\mu^Q} "10" ; "11"};
                    {\ar_{\mu^P} "00" ; "01"};
                    {\ar_{F_{\Sigma k_i}} "01" ; "11"};
                    {\ar@{=>}^{\psi_{k_1,\ldots,k_n}} (15,-5.5) ; (15,-9.5)};
                \endxy
            \]
These isomorphisms are then required to satisfy their own axioms, and these ensure that we have a weak map of 2-monads $\underline{P} \Rightarrow \underline{Q}$.  In particular, one can show that a pseudomorphism from $P$ into the endomorphism operad $\mathcal{E}_X$ produces pseudoalgebras for the operad $P$ using the closed structure on $\mb{Cat}$.  While abstractly pleasing, we do not pursue this argument any further here.
\end{rem}

\section{Basic properties}
This section will be concerned with characterizing various properties of those $2$-monads induced by $\mb{G}$-operads. We first consider when these $2$-monads are finitary as this describes how they interact with colimits. We will then give conditions for these $2$-monads to be $2$-cartesian, describing how they interact with certain limits (namely $2$-pullbacks). Finally in this section we will continue the study of algebras for these $2$-monads, showing that the coherence theorem in \cite{lack-cod} applies to all such $2$-monads and allows us to show that each pseudo-$\underline{P}$-algebra is equivalent to a strict $\underline{P}$-algebra (and so similarly, by our previous results, to the pseudoalgebras for a $\mb{G}$-operad $P$).

 The $2$-categories $\mb{Ps}\mbox{-}T\mbox{-}\mb{Alg}$ (of pseudoalgebras and weak morphisms) and $T\mbox{-}\mb{Alg}_s$ (of strict algebras and strict morphisms) are of particular interest. The behavior of colimits in both of these 2-categories can often be deduced from properties of the 2-monad $T$, the most common being that $T$ is finitary.  In practice, one thinks of a finitary monad as one in which all operations take finitely many inputs as variables.  If $T$ is finitary, then $T\mbox{-}\mb{Alg}_s$ will be cocomplete by standard results given in \cite{BKP}.  There are additional results of a purely 2-dimensional nature concerning finitary $2$-monads, detailed in \cite{lack-cod} and extending those in \cite{BKP}, namely the existence of a left adjoint
    \[
        \mb{Ps}\mbox{-}T\mbox{-}\mb{Alg} \rightarrow T\mbox{-}\mb{Alg}_s
    \]
to the forgetful $2$-functor which regards a strict algebra as a pseudoalgebra with identity structure isomorphisms.

We begin by showing each associated $2$-monad is finitary.
\begin{prop}
Let $P$ be a $\mb{G}$-operad. Then $\underline{P}$ is finitary.
\end{prop}
\begin{proof}
To show that $\underline{P}$ is finitary we must show that it preserves filtered colimits or, equivalently, that it preserves directed colimits (see \cite{ar}). Consider some directed colimit, $\text{colim}X_{i}$ say, in $\mathbf{Cat}$. Then consider the following sequence of isomorphisms:
    \begin{align*}
        \underline{P}(\text{colim}X_{i}) &= \coprod_n P(n) \times_{G(n)} (\text{colim}X_{i})^n \\
                                                               &\cong \coprod_n P(n) \times_{G(n)} \text{colim}(X_{i}^n) \\
                                                               &\cong \coprod_n \text{colim}(P(n) \times_{G(n)} X_{i}^n) \\
                                                               &\cong \text{colim}\coprod_n P(n) \times_{G(n)} X_{i}^n = \text{colim}\underline{P}(X_{i}).
    \end{align*}
Since $\mathbf{Cat}$ is locally finitely presentable then directed colimits commute with finite limits, giving the first isomorphism. The second isomorphism follows from this fact as well as that colimits commute with coequalizers. The third isomorphism is simply coproducts commuting with other colimits.
\end{proof}

The monads arising from a non-symmetric operad are always cartesian, as described in \cite{leinster}. The monads that arise from symmetric operads, however, are not always cartesian and so it is useful to be able to characterize exactly when they are. An example of where this fails is the symmetric operad for which the algebras are commutative monoids. In the case of $2$-monads we can consider the  strict $2$-limit analogous to the pullback, the $2$-pullback, and characterize when the induced $2$-monad from a $\mb{G}$-operad is $2$-cartesian, as we now describe.

\begin{Defi}
A $2$-monad $T \colon \mathcal{K} \rightarrow \mathcal{K}$ is said to be \textit{$2$-cartesian} if
    \begin{itemize}
        \item the $2$-category $\mathcal{K}$ has $2$-pullbacks,
        \item the functor $T$ preserves $2$-pullbacks, and
        \item the naturality squares for the unit and multiplication of the $2$-monad are $2$-pullbacks.
    \end{itemize}
\end{Defi}

It is important to note that the  $2$-pullback of a diagram is actually the same as the ordinary pullback in $\mb{Cat}$, see \cite{kelly-elem}. Since we will be computing with coequalizers of the form $A \times_{G} B$ repeatedly, we give the following useful lemma.

\begin{lem}\label{coeq-lem}
Let $G$ be a group and let $A$, $B$ be categories for which $A$ has a right action by $G$ and $B$ has a left action by $G$. There is then an action of $G$ on the product $A \times B$ given by
    \[
        (a,b) \cdot g \colon= (a \cdot g, g^{-1} \cdot b).
    \]
The category $A \times B/G$, consisting of the equivalence classes of this action, is isomorphic to the coequalizer $A \times_G B$.
\end{lem}
\begin{proof}
The category $A \times_G B$ is defined as the coequalizer
    \[
        \xy
            (0,0)*+{A \times G \times B}="00";
            (30,0)*+{A \times B}="10";
            (60,0)*+{A \times_G B}="20";
            {\ar@<1ex>^{\lambda} "00" ; "10"};
            {\ar@<-1ex>_{\rho} "00" ; "10"};
            {\ar^{\varepsilon} "10" ; "20"};
        \endxy
    \]
where $\lambda(a,g,b) = (a \cdot g, b)$ and $\rho(a,g,b) = (a, g \cdot b)$. However, the map $A \times B \rightarrow A \times B/G$, sending $(a,b)$ to the equivalence class $[a,b] = [a \cdot g, g^{-1} \cdot b]$, also coequalizes $\lambda$ and $\rho$ since
    \[
        [a \cdot g, b] = [(a \cdot g) \cdot g^{-1}, g \cdot b] = [a, g \cdot b].
    \]

Given any other category $X$ and a functor $\chi \colon A \times B \rightarrow X$ which coequalizes $\lambda$ and $\rho$, we get a functor $\phi \colon A \times B/G \rightarrow X$ defined by $\phi[a,b] = \chi(a,b)$. That this is well defined is clear, since
    \[
        \phi[a \cdot g, g^{-1} \cdot b] = \chi(a \cdot g, g^{-1} \cdot b) = \chi(a \cdot (gg^{-1}), b) = \chi(a, b) = \phi[a,b].
    \]
This is also unique and so we find that $A \times B/G$ satisfies the universal property of the coequalizer.
\end{proof}

\begin{prop}
Let $P$ be a $\mb{G}$-operad. Then the $2$-monad $\underline{P}$ is $2$-cartesian if and only if the action of each $G(n)$ on $P(n)$ has the following property:
    \begin{itemize}
        \item if $p \in P(n)$ and $g \in G(n)$ such that $p \cdot g = p$, then $g \in ker \pi_n$, where $\pi_n \colon G(n) \rightarrow \Sigma_n$.
    \end{itemize}
\end{prop}
\begin{proof}
Consider the following pullback of discrete categories.
    \[
        \xy
            (0,0)*+{\lbrace (x,y), (x,y'), (x',y), (x',y') \rbrace}="00";
            (40,0)*+{\lbrace y,y' \rbrace}="10";
            (0,-15)*+{\lbrace x, x' \rbrace}="01";
            (40,-15)*+{\lbrace z \rbrace}="11";
            {\ar "00" ; "10"};
            {\ar "10" ; "11"};
            {\ar "00" ; "01"};
            {\ar "01" ; "11"};
        \endxy
    \]
Letting $\mathbf{4}$ denote the pullback and similarly writing $\mathbf{2}_X = \{ x, x' \}$ and $\mathbf{2}_Y = \{y, y'\}$, we get the following diagram as the image of this pullback square under $\underline{P}$.
    \[
        \xy
            (0,0)*+{\coprod P(n) \times_{G(n)} \mathbf{4}^n}="00";
            (40,0)*+{\coprod P(n) \times_{G(n)} \mathbf{2}_Y^n}="10";
            (0,-15)*+{\coprod P(n) \times_{G(n)} \mathbf{2}_X^n}="01";
            (40,-15)*+{\coprod P(n)/G(n)}="11";
            {\ar "00" ; "10"};
            {\ar "10" ; "11"};
            {\ar "00" ; "01"};
            {\ar "01" ; "11"}:
        \endxy
    \]
The projection map $\underline{P}(\mb{4}) \rightarrow \underline{P}(\mb{2}_Y)$ maps an element
    \[
        [p;(x_1,y_1), \ldots, (x_n,y_n)]
    \]
to the element
    \[
        [p;y_1,\ldots,y_n]
    \]
and likewise for the projection to $\underline{P}(\mb{2}_X)$.

Now assume, in order to derive a contradiction, that, for some $n$, that the action of $G(n)$ on $P(n)$ does not have the prescribed property. Then find some $p \in P(n)$ along with $g \notin \text{ker} \, \pi_n$ such that $p \cdot g = p$. We will show that the existence of $g$ proves that $\underline{P}$ is not cartesian.

Now $\pi(g) \neq e$ since $g$ is not in the kernel, so there exists an $i$ such that $\pi(g)(i) \neq i$; without loss of generality, we may take $i=1$. Using this $g$ we can find two distinct elements
    \[
        [p;(x',y),(x,y),\ldots,(x,y),(x,y'),(x,y),\ldots,(x,y)]
    \]
and
    \[
        [p;(x,y),\ldots,(x,y),(x',y'),(x,y),\ldots,(x,y)]
    \]
in $\underline{P}(\mb{4})$.  In the first element we put $(x',y)$ in the first position and $(x,y')$ in position $\pi(g)(1)$, whilst in the second element we put $(x',y')$ in position $\pi(g)(1)$. Both of these elements, however, are mapped to the same elements in $\underline{P}(\mb{2}_X)$, since
    \begin{align*}
           [p; x', x, \ldots, x]&= [p \cdot g; (x', x, \ldots, x)]\\
          &= [p;\pi(g)\cdot (x', x, \ldots, x)]\\
          &= [p;x,x,\ldots,x',\ldots,x].
    \end{align*}
Similarly, both of the elements are mapped to the same element in $\underline{P}(\mathbf{2}_Y)$, simply
    \[
        [p;y,\ldots,y', \ldots, y].
    \]
The pullback of this diagram, however, has a unique element which is projected to the ones we have considered, so $\underline{P}(\mb{4})$ is not a pullback. Hence $\underline{P}$ does not preserve pullbacks if for some $n$ the action of $G(n)$ on $P(n)$ does not have the given property.

For the rest of the proof we will assume that each $G(n)$ acts on $P(n)$ in the prescribed way. We require that the naturality squares for $\eta$ and $\mu$ are $2$-pullbacks. In the case of $\eta$ this is to require that for a functor $f \colon X \rightarrow Y$, the pullback of the following diagram is the category $X$.
	\[
		\xy
			(40,0)*+{Y}="10";
			(0,-15)*+{\coprod P(n) \times_{G(n)} X^n}="01";
			(40,-15)*+{\coprod P(n) \times_{G(n)} Y^n}="11";
			{\ar^{\eta_Y} "10" ; "11"};
			{\ar_{\underline{P}(f)} "01" ; "11"};
		\endxy
	\]
The pullback of this diagram is isomorphic to the coproduct of the pullbacks of diagrams of the following form.
\[
		\xy
			(30,0)*+{Y}="10";
			(0,-15)*+{P(n) \times_{G(n)} X^n}="01";
			(30,-15)*+{P(n) \times_{G(n)} Y^n}="11";
			{\ar^{} "10" ; "11"};
			{\ar_{1 \times f^n} "01" ; "11"};
		\endxy
	\]

Note that $P(1) \times_{G(1)} Y$ is isomorphic to $(P(1)/G(1)) \times Y$, the latter clearly satisfying the universal property of the coequalizer - since every element of $G(1)$ acts trivially on $Y$ we can write objects of $P(1) \times_{G(1)} Y$ as pairs $([p],y)$, where $p \in P(1)$ and $y \in Y$.

Now since $\eta_Y$ lands in $P(1) \times_{G(1)} Y$, we need only check that $X$ is the pullback of the above cospan in the case that $n = 1$. The pullback is then the category consisting of pairs $(([p],x),y)$ such that $([p],f(x)) = ([id],y)$. Such pairs exist only when $y = f(x)$ and $[p] = [id]$, showing that $X$ is indeed the pullback. Thus naturality squares for $\eta$ are pullbacks.

For $\mu$ we will use the fact that if all of the diagrams
    \[
        \xy
            (0,0)*+{\underline{P}^2(X)}="00";
            (20,0)*+{\underline{P}^2(1)}="10";
            (0,-15)*+{\underline{P}(X)}="01";
            (20,-15)*+{\underline{P}(1)}="11";
            {\ar^{\underline{P}^2(!)} "00" ; "10"};
            {\ar^{\mu_1} "10" ; "11"};
            {\ar_{\mu_X} "00" ; "01"};
            {\ar_{\underline{P}(!)} "01" ; "11"};
        \endxy
    \]
are pullbacks then the outside of the diagram
    \[
        \xy
            (0,0)*+{\underline{P}^2(X)}="00";
            (20,0)*+{\underline{P}^2(Y)}="10";
            (40,0)*+{\underline{P}^2(1)}="20";
            (0,-15)*+{\underline{P}(X)}="01";
            (20,-15)*+{\underline{P}(Y)}="11";
            (40,-15)*+{\underline{P}(1)}="21";
            {\ar^{\underline{P}^2(f)} "00" ; "10"};
            {\ar^{\underline{P}^2(!)} "10" ; "20"};
            {\ar^{\mu_{1}} "20" ; "21"};
            {\ar_{\mu_X} "00" ; "01"};
            {\ar_{\underline{P}(f)} "01" ; "11"};
            {\ar_{\underline{P}(!)} "11" ; "21"};
            {\ar_{\mu_Y} "10" ; "11"};
        \endxy
    \]
is also a pullback and so each of the naturality squares for $\mu$ must therefore be a pullback. Now we can split up the square above, much like we did for $\eta$, and prove that each of the squares
    \[
        \xy
            (0,0)*+{\coprod P(m) \times_{G(m)} \prod_i \left(P(k_i) \times_{G(k_i)} X^{k_i}\right)}="00";
            (60,0)*+{\coprod P(m) \times_{G(m)} \prod_i \left(P(k_i) / G(k_i)\right)}="10";
            (0,-20)*+{P(n) \times_{G(n)} X^n}="01";
            (60,-20)*+{P(n) / G(n)}="11";
            {\ar "00" ; "10"};
            {\ar "00" ; "01"};
            {\ar "01" ; "11"};
            {\ar "10" ; "11"};
        \endxy
    \]
is a pullback. The map along the bottom is the obvious one, sending $[p; x_1, \ldots, x_n]$ simply to the equivalence class $[p]$. Along the right hand side the map is the one corresponding to operadic composition, sending $[q;[p_1],\ldots,[p_m]]$ to $[\mu^P(q;p_1,\ldots,p_n)]$. The pullback of these maps would be the category consisting of pairs
    \[
        \left([p;x_1,\ldots,x_{\Sigma k_i}],[q;[p_1],\ldots,[p_n]]\right),
    \]
where $q \in P(n)$, $p_i \in P(k_i)$, $p \in P(\Sigma k_i)$, and for which $[p] = [\mu^P(q;p_1,\ldots,p_n)]$. The upper left category in the diagram, which we will refer to here as $Q$, has objects
    \[
        [q;[p_1;\underline{x}_1],\ldots,[p_n;\underline{x}_n]].
    \]

There are obvious maps out of $Q$ making the diagram commute and as such inducing a functor from $Q$ into the pullback via the universal property. This functor sends an object such as the one just described to the pair
    \[
        \left([\mu^P(q;p_1,\ldots,p_n);\underline{x}], [q;[p_1],\ldots,[p_n]]\right).
    \]
Given an object in the pullback, we then have a pair, as described above, which has $[p] = [\mu^P(q;p_1,\ldots,p_n)]$ meaning that we can find an element $g \in G(\Sigma k_i)$ such that $p  = \mu^P(q;p_1,\ldots,p_n) \cdot g$. Thus we can describe an inverse to the induced functor by sending a pair in the pullback to the object
    \[
        [q;[p_1;\pi(g)(\underline{x})_1],\ldots,[p_n;\pi(g)(\underline{x})_n]],
    \]
where $\pi(g)(\underline{x})_i$ denotes the $i$th block of $\underline{x}$ after applying the permutation $\pi(g)$. For example, if $\underline{x} = (x_{11}, x_{12}, x_{21}, x_{22}, x_{23}, x_{31})$ and $\pi(g) = (1\, 3\, 5)$, then $\pi(g)(\underline{x}) = (x_{23}, x_{12}, x_{11}, x_{22}, x_{21}, x_{31})$. Thus $\pi(g)(\underline{x})_1 = (x_{23}, x_{12})$, $\pi(g)(\underline{x})_2 = (x_{11}, x_{22}, x_{21})$ and $\pi(g)(\underline{x})_3 = (x_{31})$. Now applying the induced functor we find that we get back an object in the pullback for which the first entry is $[q;[p_1],\ldots,[p_n]]$ and whose second entry is
    \[
       [\mu^P(q;p_1,\ldots,p_n);\pi(g)(\underline{x})] = [\mu^P(q;p_1,\ldots,p_n) \cdot g;\underline{x}] = [p;\underline{x}],
    \]
which is what we started with. Showing the other composite is an identity is similar, here using the fact that the identity acts trivially on $\mu^P(q;p_1,\ldots,p_n)$. Taking the coproduct of these squares then gives us the original diagram that we wanted to show was a pullback and, since each individual square is a pullback, so is the original.

To finish we must also show that $\underline{P}$ preserves pullbacks. Given a pullback
    \[
        \xy
            (0,0)*+{A}="00";
            (15,0)*+{B}="10";
            (0,-15)*+{C}="01";
            (15,-15)*+{D}="11";
            {\ar^{F} "00" ; "10"};
            {\ar^{S} "10" ; "11"};
            {\ar_{R} "00" ; "01"};
            {\ar_{H} "01" ; "11"};
        \endxy
    \]
we must show that the image of the diagram under $\underline{P}$ is also a pullback. Now this will be true if and only if each individual diagram
        \[
            \xy
                (0,0)*+{P(n) \times_{G(n)} A^n}="00";
                (30,0)*+{P(n) \times_{G(n)} B^n}="10";
                (0,-15)*+{P(n) \times_{G(n)} C^n}="01";
                (30,-15)*+{P(n) \times_{G(n)} D^n}="11";
                {\ar^{1 \times F^n} "00" ; "10"};
                {\ar^{1 \times S^n} "10" ; "11"};
                {\ar_{1 \times R^n} "00" ; "01"};
                {\ar_{1 \times H^n} "01" ; "11"}:
            \endxy
    \]
is also a pullback. The pullback of the functors $1 \times H^n$ and $1 \times S^n$ is a category consisting of pairs of objects $[p;\underline{c}]$ and $[q;\underline{b}]$, where $\underline{b}$ and $\underline{c}$ represent lists of elements in $B$ and $C$, respectively. These pairs are then required to satisfy the property that
    \[
        [p;\underline{H(c)}] = [q; \underline{S(b)}].
    \]
Using the previous lemma, we know that a pair
    \[
        \left([p;\underline{c}], [q;\underline{b}]\right)
    \]
is in the pullback if and only if there exists an element $g \in G(n)$ such that $p \cdot g = q$ and $Hc_i = (Sb_{\pi(g)^{-1}(i)})$. Using this we can define mutual inverses between $P(n) \times_{G(n)} A^n$ and the pullback $Q'$. Considering the category $A$ as the pullback of the diagram we started with, we can consider objects of $P(n) \times_{G(n)} A^n$ as being equivalence classes
    \[
        [p;(b_1,c_1),\ldots,(b_n,c_n)]
    \]
where $p \in P(n)$ and $Hc_i = Sb_i$ for all $i$.

Taking such an object, we send it to the pair
    \[
        \left([p;c_1,\ldots,c_n],[p;b_1,\ldots,b_n]\right)
    \]
which lies in the pullback since the identity in $G(n)$ satisfies the condition given earlier. An inverse to this sends a pair of equivalence classes in $Q'$ to the single equivalence class
    \[
        [p;(c_1,b_{\pi(g)^{-1}(1)}),\ldots,(c_n,b_{\pi(g)^{-1}(n)})]
    \]
in $P(n) \times_{G(n)} A^n$. If we apply the map into $Q'$ we get the pair
    \[
        \left([p;c_1,\ldots,c_n],[p;b_{\pi(g)^{-1}(1)},\ldots,b_{\pi(g)^{-1}(n)}]\right)
    \]
which is equal to the original pair since $p \cdot g = q$. The other composite is trivially an identity since the identity in $G(n)$ has trivial permutation.
\end{proof}
\begin{cor}
The $2$-monad associated to a symmetric operad $P$ is $2$-cartesian if and only if the action of $\Sigma_n$ is free on each $P(n)$.
\end{cor}

The final part of this section is motivated by the issue of coherence. At its most basic, a coherence theorem is a way of describing when a notion of weaker structure is in some way equivalent to a stricter structure. The prototypical case here is the coherence theorem for monoidal categories. In a monoidal category we require associator isomorphisms
    \[
        \left( A \otimes B \right) C \cong A \otimes \left( B \otimes C \right)
    \]
for all objects in the category. The coherence theorem tells us that, for any monoidal category $M$, there is a strict monoidal category which is equivalent to $M$.  In other words, we can treat the associators in $M$ as identities, and similarly for the unit isomorphisms.

The abstract approach to coherence considers when the pseudoalgebras for a $2$-monad $T$ are equivalent to strict $T$-algebras, with the most comprehensive account appearing in \cite{lack-cod}.  Lack gives a general theorem which provides sufficient conditions for the existence of a left adjoint to the forgetful $2$-functor
    \[
        U \colon T\mbox{-}\mb{Alg}_s \rightarrow \mb{Ps}\mbox{-}T\mbox{-}\mb{Alg}
    \]
for which the components of the unit of the adjunction are equivalences. We focus on one version of this general result which has hypotheses that are quite easy to check in practice.  First we require that the base 2-category $\mathcal{K}$ has an enhanced factorization system. This is much like an orthogonal factorization system on a $2$-category, consisting of two classes of maps $(\mathcal{L},\mathcal{R})$, satisfying the lifting properties on $1$-cells and $2$-cells as follows. Given a commutative square
     \[
        \xy
            (0,0)*+{A}="00";
            (15,0)*+{C}="10";
            (0,-15)*+{B}="01";
            (15,-15)*+{D}="11";
            {\ar^{f} "00" ; "10"};
            {\ar^{r} "10" ; "11"};
            {\ar_{l} "00" ; "01"};
            {\ar_{g} "01" ; "11"};
        \endxy
     \]
where $l \in \m{L}$ and $r \in {R}$, there exists a unique morphism $m \colon B \rightarrow C$ such that $rm = g$ and $ml = f$. Similarly, given two commuting squares for which $rf = gl$ and $rf' = f'l$, along with $2$-cells $\delta \colon f \Rightarrow f'$ and $\gamma \colon g \Rightarrow g'$ for which $\gamma \ast 1_l = 1_r \ast \delta$, there exists a unique $2$-cell $\mu \colon m \Rightarrow m'$, where $m$ and $m'$ are induced by the $1$-cell lifting property, satisfying $\mu \ast 1_l = \delta$ and $1_r \ast \mu = \gamma$. However, there is an additional $2$-dimensional property of the factorization system which says that given maps $l \in \m{L}$, $r \in \m{R}$ and an invertible $2$-cell $\alpha \colon rf \Rightarrow gl$
    \[
        \xy
            (0,0)*+{A}="00";
            (15,0)*+{C}="10";
            (0,-15)*+{B}="01";
            (15,-15)*+{D}="11";
            {\ar^{f} "00" ; "10"};
            {\ar^{r} "10" ; "11"};
            {\ar_{l} "00" ; "01"};
            {\ar_{g} "01" ; "11"};
            {\ar@{=>}^{\alpha} (9.375,-5.625) ; (5.625,-9.375)};
            (22.5,-7.5)*+{=};
            (30,0)*+{A}="20";
            (45,0)*+{C}="30";
            (30,-15)*+{B}="21";
            (45,-15)*+{D}="31";
            {\ar^{f} "20" ; "30"};
            {\ar^{r} "30" ; "31"};
            {\ar_{l} "20" ; "21"};
            {\ar_{g} "21" ; "31"};
            {\ar^{m} "21" ; "30"};
            {\ar@{=>}^{\beta} (41,-8) ; (38,-12)};
        \endxy
    \]
there is a unique pair $(m,\beta)$ where $m \colon C \rightarrow B$ is a $1$-cell and $\beta \colon rm \Rightarrow g$ is an invertible $2$-cell such that $ml = f$ and $\beta \ast 1_{l} = \alpha$.

Further conditions require that $T$ preserve $\mathcal{L}$ maps and that whenever $r \in \mathcal{R}$ and $rk \cong 1$, then $kr \cong 1$. In our case we are considering $2$-monads on the $2$-category $\mathbf{Cat}$, which has the enhanced factorization system where $\m{L}$ consists of bijective-on-objects functors and $\m{R}$ is given by the full and faithful functors. This, along with the $2$-dimensional property making it an enhanced factorization system, is described in \cite{power-gen}. The last stated condition, involving isomorphisms and maps in $\m{R}$, is then clearly satisfied and so the only thing we need to check in order to satisfy the conditions of the coherence result are that the induced $2$-monads $\underline{P}$ preserve bijective-on-objects functors, which is a simple exercise involving coequalizers.

\begin{prop}
For any $\mb{G}$-operad $P$, the $2$-monad $\underline{P}$ preserves bijective-on-objects functors.
\end{prop}
\begin{cor}
Every pseudo-$\underline{P}$-algebra is equivalent to a strict $\underline{P}$-algebra.
\end{cor} 

\section{Pseudo-commutativity}

This final section gives conditions sufficient to equip the 2-monad $\underline{P}$ induced by a $\mb{G}$-operad $P$ in $\mb{Cat}$ with a pseudo-commutative structure.  Such a pseudo-commutativity will then give the 2-category $\mb{Ps}\mbox{-}\underline{P}\mbox{-}\mb{Alg}$ some additional structure that we briefly explain here.  For a field $k$, the category $\mb{Vect}$ of vector spaces over $k$ has many nice features.  Of particular interest to us are the following three structures.  First, the category $\mb{Vect}$ is monoidal using the tensor product $\otimes_{k}$.  Second, the set of linear maps $V \rightarrow W$ is itself a vector space which we denote $[V,W]$.  Third, there is a notion of multilinear map $V_{1} \times \cdots \times V_{n} \rightarrow W$, with linear maps being the 1-ary version.  While these three structures are each useful in isolation, they are tied together by natural isomorphisms
\[
\mb{Vect}(V_{1} \otimes V_{2}, W) \cong \mb{Vect}(V_{1}, [V_{2}, W]) \cong \mb{Bilin}(V_{1} \times V_{2}, W)
\]
expressing that $\otimes$ gives a closed monoidal structure which represents the multicategory of multilinear maps.  Moreover, the adjunction between $\mb{Vect}$ and $\mb{Sets}$ respects all of this structure in the appropriate way.  This incredibly rich interplay between the tensor product, the internal mapping space, and the multicategory of multilinear maps all arises from the free vector space monad on $\mb{Sets}$ being a \textit{commutative} monad \cite{kock-monads, kock-closed, kock-strong}.  The notion of a pseudo-commutative 2-monad \cite{HP} is then a generalization of this machinery to a 2-categorical context, and can be viewed as a starting point for importing tools from linear algebra into category theory.

The aim of this section is to give conditions that ensure that the 2-monad $\underline{P}$ associated to a $\mb{G}$-operad $P$ has a pseudo-commutative structure. We give the definition of pseudo-commutativity as in \cite{HP} but before doing so we note what we mean by a strength for a $2$-monad.
\begin{Defi}
A \textit{strength} for an endo-$2$-functor $T \colon \m{K} \rightarrow \m{K}$ on a 2-category with products and terminal object $1$ consists of a $2$-natural transformation $t$ with components
    \[
        t_{A,B} \colon A \times TB \rightarrow T(A \times B)
    \]
satisfying the following unit and associativity axioms \cite{kock-strong}.
\[
\xy
(0,0)*+{1 \times TA}="ul1";
(30,0)*+{T(1 \times A)}="ur1";
(30,-13)*+{TA}="br1";
(50,0)*+{A \times B}="ul2";
(80,0)*+{A \times TB}="ur2";
(80,-13)*+{T(A \times B)}="br2";
{\ar^{t_{1,A}} "ul1"; "ur1"};
{\ar^{\cong} "ur1"; "br1"};
{\ar_{\cong} "ul1"; "br1"};
{\ar^{1 \times \eta} "ul2"; "ur2"};
{\ar^{t_{A,B}} "ur2"; "br2"};
{\ar_{\eta} "ul2"; "br2"};
\endxy
\]
\[
\xy
(0,0)*+{(A \times B) \times TC}="ul";
(70,0)*+{T \Big((A \times B) \times C \Big)}="ur";
(0,-15)*+{A \times (B \times TC)}="ll";
(35,-15)*+{A \times T(B \times C)}="m";
(70,-15)*+{ T \Big(A \times (B \times C) \Big)}="lr";
{\ar^{t_{AB,C}} "ul"; "ur"};
{\ar^{Ta} "ur"; "lr"};
{\ar_{a} "ul"; "ll"};
{\ar_{1 \times t_{B,C}} "ll"; "m"};
{\ar_{t_{A,BC}} "m"; "lr"};
\endxy
\]
\[
\xy
(0,0)*+{A \times T^{2}B}="ul";
(60,0)*+{T^{2}(A \times B)}="ur";
(0,-15)*+{A \times TB}="ll";
(30,0)*+{T(A \times TB)}="m";
(60,-15)*+{ T(A \times B)}="lr";
{\ar^{t_{A,TB}} "ul"; "m"};
{\ar^{Tt_{A,B}} "m"; "ur"};
{\ar^{\mu} "ur"; "lr"};
{\ar_{1 \times \mu} "ul"; "ll"};
{\ar_{t_{A,B}} "ll"; "lr"};
\endxy
\]
Similarly, a costrength for $T$ consists of a $2$-natural transformation $t^{\ast}$ with components
    \[
        t^{\ast}_{A,B} \colon TA \times B \rightarrow T(A \times B)
    \]
again satisfying unit and associativity axioms.
\end{Defi}
The strength and costrength for the associated $2$-monad $\underline{P}$ are quite simple to define. We define the strength $t$ for $\underline{P}$ as follows. The component $t_{A,B}$ is a functor
    \[
        t_{A,B} \colon A \times \left(\amalg P(n) \times_{G(n)} B^n\right) \rightarrow \amalg P(n) \times_{G(n)} \left(A \times B \right)^n
    \]
which sends an object $(a, [p;b_1,\ldots,b_n])$ to the object $[p;(a,b_1),\ldots,(a,b_n)]$. We also define the costrength similarly, sending an object $([p;a_1,\ldots,a_n],b)$ to the object $[p;(a_1,b), \ldots, (a_n, b)]$. Both the strength and the costrength are defined in the obvious way on morphisms.
\begin{Defi}
    Given a $2$-monad $T \colon \m{K} \rightarrow \m{K}$ with strength $t$ and costrength $t^{\ast}$, a \textit{pseudo-commutativity} consists of an invertible modification $\gamma$ with components
        \[
            \xy
                (0,0)*+{TA \times TB}="00";
                (30,0)*+{T(A \times TB)}="10";
                (60,0)*+{T^2(A \times B)}="20";
                (0,-15)*+{T(TA \times B)}="01";
                (30,-15)*+{T^2(A \times B)}="11";
                (60,-15)*+{T(A \times B)}="21";
                {\ar^{t^{\ast}_{A,TB}} "00" ; "10"};
                {\ar^{Tt_{A,B}} "10" ; "20"};
                {\ar^{\mu_{A \times B}} "20" ; "21"};
                {\ar_{t_{TA,B}} "00" ; "01"};
                {\ar_{Tt^{\ast}_{A,B}} "01" ; "11"};
                {\ar_{\mu_{A \times B}} "11" ; "21"};
                {\ar@{=>}^{\gamma_{A,B}} (30,-5.5) ; (30,-9.5)};
            \endxy
        \]
satisfying the following three strength axioms, two unit (or $\eta$) and two multiplication (or $\mu$) axioms for all $A$, $B$, and $C$.
    \begin{enumerate}
        \item $\gamma_{A \times B,C} \circ (t_{A,B} \times 1_{TC}) = t_{A,B \times C} \circ (1_A \times \gamma_{B,C})$
        \item $\gamma_{A,B \times C} \circ (1_{TA} \times t_{B,C}) = \gamma_{A \times B, C} \circ (t^{\ast}_{A,B} \times 1_{TC})$
        \item $\gamma_{A,B \times C} \circ (1_{TA} \times T^{\ast}_{B,C}) = t^{\ast}_{A \times B,C} \circ (\gamma_{A,B} \times 1_{C})$
        \item $\gamma_{A,B} \circ (\eta_A \times 1_{TB})$  is an identity.
        \item $\gamma_{A,B} \circ (1_{TA} \times \eta_B)$ is an identity.
        \item $\gamma_{A,B} \circ (\mu_A \times 1_{TB})$ is equal to the pasting below.
            \[
                \xy
                    (0,0)*+{\scriptstyle T^2A \times TB}="00";
                    (30,0)*+{\scriptstyle T(TA \times TB)}="10";
                    (60,0)*+{\scriptstyle T^2(A \times TB)}="20";
                    (90,0)*+{\scriptstyle T^3(A \times B)}="30";
                    (0,-15)*+{\scriptstyle T(T^2A \times B)}="01";
                    (30,-15)*+{\scriptstyle T^2(TA \times B)}="11";
                    (60,-15)*+{\scriptstyle T^3(A \times B)}="21";
                    (90,-15)*+{\scriptstyle T^2(A \times B)}="31";
                    (0,-30)*+{\scriptstyle T^2(TA \times B)}="02";
                    (30,-30)*+{\scriptstyle T(TA \times B)}="12";
                    (60,-30)*+{\scriptstyle T^2(A \times B)}="22";
                    (90,-30)*+{\scriptstyle T(A \times B)}="32";
                    {\ar^{t^{\ast}_{TA,TB}} "00" ; "10"};
                    {\ar^{Tt^{\ast}_{A,TB}} "10" ; "20"};
                    {\ar^{T^2t_{A,B}} "20" ; "30"};
                    {\ar_{t_{T^2A,B}} "00" ; "01"};
                    {\ar_{Tt_{TA,B}} "10" ; "11"};
                    {\ar^{T\mu_{A \times B}} "30" ; "31"};
                    {\ar_{T^2t^{\ast}_{A,B}} "11" ; "21"};
                    {\ar_{T\mu_{A \times B}} "21" ; "31"};
                    {\ar_{Tt^{\ast}_{TA,B}} "01" ; "02"};
                    {\ar_{\mu_{TA \times B}} "11" ; "12"};
                    {\ar_{\mu_{T(A \times B)}} "21" ; "22"};
                    {\ar^{\mu_{A \times B}} "31" ; "32"};
                    {\ar_{\mu_{TA \times B}} "02" ; "12"};
                    {\ar_{Tt^{\ast}_{A,B}} "12" ; "22"};
                    {\ar_{\mu_{A \times B}} "22" ; "32"};
                    {\ar@{=>}^{T\gamma_{A,B}} (60,-5.5) ; (60,-9.5)};
                    {\ar@{=>}^{\gamma_{TA,B}} (12.5,-13) ; (12.5,-17)};
                \endxy
            \]
        \item $\gamma_{A,B} \circ (1_{TA} \times \mu_B)$ is equal to the pasting below.
                    \[
                \xy
                    (0,0)*+{\scriptstyle TA \times T^2B}="00";
                    (30,0)*+{\scriptstyle T(A \times T^2B)}="10";
                    (60,0)*+{\scriptstyle T^2(A \times TB)}="20";
                    (0,-15)*+{\scriptstyle T(TA \times TB)}="01";
                    (30,-15)*+{\scriptstyle T^2(A \times TB)}="11";
                    (60,-15)*+{\scriptstyle T(A \times TB)}="21";
                    (0,-30)*+{\scriptstyle T^2(TA \times B)}="02";
                    (30,-30)*+{\scriptstyle T^3(A \times B)}="12";
                    (60,-30)*+{\scriptstyle T^2(A \times B)}="22";
                    (0,-45)*+{\scriptstyle T^3(A \times B)}="03";
                    (30,-45)*+{\scriptstyle T^2(A \times B)}="13";
                    (60,-45)*+{\scriptstyle T(A \times B)}="23";
                    {\ar^{t^{\ast}_{A,T^2B}} "00" ; "10"};
                    {\ar^{Tt_{A,TB}} "10" ; "20"};
                    {\ar_{t_{TA,TB}} "00" ; "01"};
                    {\ar^{\mu_{A \times TB}} "20" ; "21"};
                    {\ar_{Tt^{\ast}_{A,TB}} "01" ; "11"};
                    {\ar_{\mu_{A \times TB}} "11" ; "21"};
                    {\ar_{Tt_{TA,B}} "01" ; "02"};
                    {\ar^{T^2t_{A,B}} "11" ; "12"};
                    {\ar^{Tt_{A,B}} "21" ; "22"};
                    {\ar^{\mu_{T(A \times B)}} "12" ; "22"};
                    {\ar_{T^2t^{\ast}_{A,B}} "02" ; "03"};
                    {\ar^{T\mu_{A \times B}} "12" ; "13"};
                    {\ar^{\mu_{A \times B}} "22" ; "23"};
                    {\ar_{T\mu_{A \times B}} "03" ; "13"};
                    {\ar_{\mu_{A \times B}} "13" ; "23"};
                    {\ar@{=>}^{T\gamma_{A,B}} (13,-28) ; (13,-32)};
                    {\ar@{=>}^{\gamma_{A,TB}} (30,-5.5) ; (30,-9.5)};
                \endxy
            \]
    \end{enumerate}
\end{Defi}
\begin{rem}
    It is noted in \cite{HP} that this definition has some redundancy and therein it is shown that any two of the strength axioms immediately implies the third. Furthermore, the three strength axioms are equivalent when the $\eta$ and $\mu$ axioms hold, as well as the following associativity axiom:
        \[
            \gamma_{A,B \times C} \circ (1_{TA} \times \gamma_{B,C}) = \gamma_{A \times B,C} \times (\gamma_{A,B} \times 1_{TC}).
        \]
\end{rem}

We need some notation before stating our main theorem.  Let $\underline{a} = a_{1}, \ldots , a_{m}$ and $\underline{b} = b_{1}, \ldots, b_{n}$ be two lists.  Then the set $\{ (a_{i}, b_{j})\}$ has $mn$ elements, and two natural lexicographic orderings.  One of these we write as $\underline{(a, \underline{b})}$, and it has the order given by
\[
(a_{p}, b_{q}) < (a_{r}, b_{s}) \textrm{ if } \left\{ \begin{array}{l} p < r, \textrm{ or } \\ p=r \textrm{ and } q < s. \end{array} \right.
\]
The other we write as $\underline{(\underline{a}, b)}$, and it has the order given by
\[
(a_{p}, b_{q}) < (a_{r}, b_{s}) \textrm{ if } \left\{ \begin{array}{l} q < s, \textrm{ or } \\ q=s \textrm{ and } p < r. \end{array} \right.
\]
The notation $(a, \underline{b})$ is meant to indicate that we have a single $a$ but a list of $b$'s, so then $\underline{(a, \underline{b})}$ would represent a list which itself consists of lists of that form. There is a unique permutation $\tau_{m,n} \in \Sigma_{mn}$ which has the property that $\tau_{m,n}(i) = j$ if the $i$th element of the ordered set $\underline{(a, \underline{b})}$ is equal to the $j$th element of the ordered set $\underline{(\underline{a}, b)}$.  By construction, we have $\tau_{n,m} = \tau_{m,n}^{-1}$. We illustrate these permutations with a couple of examples.
    
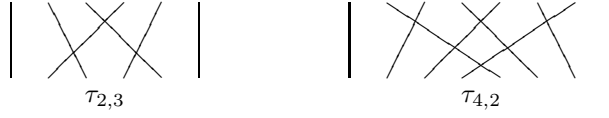
\[
        \xy
            {\ar@{-} (0,0) ; (0,-10)};
            {\ar@{-} (5,0) ; (10,-10)};
            {\ar@{-} (10,0) ; (20,-10)};
            {\ar@{-} (15,0) ; (5,-10)};
            {\ar@{-} (20,0) ; (15,-10)};
            {\ar@{-} (25,0) ; (25,-10)};
            (12.5,-13)*{\tau_{2,3}};
            {\ar@{-} (45,0) ; (45,-10)};
            {\ar@{-} (50,0) ; (65,-10)};
            {\ar@{-} (55,0) ; (50,-10)};
            {\ar@{-} (60,0) ; (70,-10)};
            {\ar@{-} (65,0) ; (55,-10)};
            {\ar@{-} (70,0) ; (75,-10)};
            {\ar@{-} (75,0) ; (60,-10)};
            {\ar@{-} (80,0) ; (80,-10)};
            (62.5,-13)*{\tau_{4,2}}
        \endxy
    \]
Note then that $\tau_{m,n}$ is the permutation given by taking the transpose of the $m \times n$ matrix with entries $(a_{i}, b_{j})$.

We now give sufficient conditions for equipping the 2-monad $\underline{P}$ associated to a $\mathbf{G}$-operad $P$ with a pseudo-commutative structure.  Let $\mathbb{N}_{+}$ denote the set of positive natural numbers.
\begin{thm}\label{pscomm}
Let $P$ be a $\mb{G}$-operad. Then the following equip $\underline{P}$ with a pseudo-commutative structure.
    \begin{enumerate}
        \item For each pair $(m,n) \in \mathbb{N}_{+}^2$, we are given an element $t_{m,n} \in G(mn)$ such that $\pi(t_{m,n}) = \tau_{m,n}$.
        \item For each $p \in P(n)$, $q \in P(m)$, we are given a natural isomorphism
            \[
                \lambda_{p,q} \colon \mu(p;q,\ldots,q) \cdot t_{m,n} \cong \mu(q;p,\ldots,p).
            \]
            We write this as $\lambda_{p,q}: \mu(p; \underline{q}) \cdot t_{m,n} \cong \mu(q; \underline{p})$.
    \end{enumerate}
These must satisfy the following:
    \begin{itemize}
        \item For all $n \in \mathbb{N}_+$,
            \[
                t_{1,n} = e_n = t_{n,1}
            \]
             and for all $p \in P(n)$, the isomorphism $\lambda_{p, \textrm{id}}: p \cdot e_n \cong p$ is the identity map.
        \item For all $l, m_1, \ldots, m_l, n \in \mathbb{N}_+$, with $M = \Sigma m_i$,
            \[
                \mu^{G}(e_l; t_{m_1,n}, \ldots, t_{m_l,n}) \cdot \mu^{G}(t_{l,n};\underline{e_{m_1},\ldots,e_{m_l}}) = t_{n,M}.
            \]
            Here $\underline{e_{m_1},\ldots,e_{m_l}}$ is the list $e_{m_{1}}, \ldots, e_{m_{l}}$ repeated $n$ times.
        \item For all $l, m, n_1,\ldots, n_m \in \mathbb{N}_+$, with $N = \Sigma n_i$,
            \[
                \mu^{G}(t_{m,l};\underline{e_{n_1}},\ldots,\underline{e_{n_m}}) \cdot \mu^{G}(e_m;t_{n_1,l},\ldots,t_{n_m,l}) = t_{N,l}.
            \]
            Here $\underline{e_{n_{i}}}$ indicates that each $e_{n_{i}}$ is repeated $l$ times.
        \item For any $l, m_i, n \in \mathbb{N}_+$, with $1 \leq i \leq n$, and $p \in P(l)$, $q_i \in P(m_i)$ and $r \in P(n)$, the following diagram commutes.  (Note that we maintain the convention that anything underlined indicates a list, and double underlining indicates a list of lists.  Each instance should have an obvious meaning from context and the equations appearing above.)
            \[
                \xy
                    (0,0)*+{\mu\Big(p;\underline{\mu(q_i;\underline{r})}\Big) \cdot \mu(e_l;\underline{t_{n,m_i}})\mu(t_{n,l};\underline{\underline{e_{m_i}}})}="00";
                    (60,0)*+{\mu\Big(p;\underline{\mu(q_i;\underline{r})}\Big) \cdot t_{n,M}}="10";
                    (0,-15)*+{\mu\Big(p;\underline{\mu(q_i;\underline{r})\cdot t_{n,m_i}}\Big) \cdot \mu(t_{n,l};\underline{e_{m_1},\ldots,e_{m_l}})}="01";
                    (60,-20)*+{\mu\Big(\mu(p;q_1,\ldots,q_n);\underline{\underline{r}}\Big)\cdot t_{n,M}}="11";
                    (0,-30)*+{\mu\Big(p;\underline{\mu(r;\underline{q_i})}\Big) \cdot \mu(t_{n,l};\underline{e_{m_1},\ldots,e_{m_l}})}="02";
                    (60,-40)*+{\mu\Big(\mu(p;q_1,\ldots,q_n);\underline{\underline{r}}\Big)}="12";
                    (0,-45)*+{\mu\Big(\mu(p;\underline{r}) \cdot t_{n,l} ; \underline{q_1,\ldots,q_n}\Big)}="03";
                    (60,-60)*+{\mu\Big(r;\underline{\mu(p;q_1,\ldots,q_n)}\Big)}="13";
                    (0,-60)*+{\mu\Big(\mu(r,\underline{p});\underline{q_1,\ldots,q_n}\Big)}="04";
                    {\ar@{=} "00" ; "10"};
                    {\ar@{=} "00" ; "01"};
                    {\ar@{=} "10" ; "11"};
                    {\ar_{\mu(1;\underline{\lambda_{q_i,r}}) \cdot 1} "01" ; "02"};
                    {\ar@{=} "02" ; "03"};
                    {\ar@{=} "04" ; "13"};
                    {\ar_{\mu(\lambda_{p,r};1)} "03" ; "04"};
                    {\ar^{\lambda_{\mu(p;q_1,\ldots,q_n),r}} "11" ; "12"};
                    {\ar@{=} "12" ; "13"};
                \endxy
            \]
        \item For any $l,m, n_i \in \mathbb{N}_+$, with $1 \leq i \leq m$, and $p \in P(l)$, $q \in P(m)$ and $r_i \in P(n_i)$, the following diagram commutes.
                \[
                    \xy
                        (0,0)*+{\mu\Big(\mu(p;\underline{q}) \cdot t_{m,l} ; \underline{\underline{r_i}}\Big) \cdot \mu(e_m;\underline{t_{n_i,l}})}="00";
                        (60,0)*+{\mu\Big(\mu(p;\underline{q});\underline{\underline{r_i}}\Big) \cdot \mu(t_{m,l};\underline{\underline{e_{n_i}}})\mu(e_{m};\underline{t_{n_i,l}})}="10";
                        (60,-15)*+{\mu\Big(p;\underline{\mu(q;\underline{r_i})}\Big) \cdot \mu(t_{m,l};\underline{\underline{e_{n_i}}})\mu(e_{m};\underline{t_{n_i,l}})}="11";
                        (0,-20)*+{\mu\Big(\mu(q;\underline{p}); \underline{r_1},\ldots,\underline{r_m}\Big) \cdot \mu(e_m;\underline{t_{n_i,l}})}="01";
                        (0,-40)*+{\mu\Big(q;\underline{\underline{\mu(p;r_i)}}\Big) \cdot \mu(e_m;\underline{t_{n_i,l}})}="02";
                        (0,-60)*+{\mu\Big(q;\underline{\mu(p;\underline{r_i}) \cdot t_{n_i,l}}\Big)}="03";
                        (60,-30)*+{\mu\Big(p;\underline{\mu(q;r_1,\ldots,r_m)}\Big) \cdot t_{N,l}}="12";
                        (60,-45)*+{\mu\Big(\mu(q;r_1,\ldots,r_m); \underline{\underline{p}}\Big)}="13";
                        (60,-60)*+{\mu\Big(q;\underline{\mu(r_i;\underline{p})}\Big)}="14";
                        {\ar@{=} "00" ; "10"};
                        {\ar@{=} "10" ; "11"};
                        {\ar@{=} "11" ; "12"};
                        {\ar^{\lambda_{p,\mu(q;r_1,\ldots,r_m)}} "12" ; "13"};
                        {\ar@{=} "13" ; "14"};
                        {\ar_{\mu(\lambda_{p,q};1) \cdot 1} "00" ; "01"};
                        {\ar@{=} "01" ; "02"};
                        {\ar@{=} "02" ; "03"};
                        {\ar_{\mu(1;\underline{\lambda_{p,r_i}})} "03" ; "14"};
                    \endxy
                \]
    \end{itemize}
\end{thm}
\begin{proof}
We begin the proof by defining an invertible modifcation $\gamma$ for the pseudo-commutativity for which the components are natural transformations $\gamma_{A,B}$.  Such a transformation $\gamma_{A,B}$ has components with source
\[
[\mu(p; \underline{q}); \underline{(x, \underline{y})}]
\]
and target
\[
[\mu(q; \underline{p}); \underline{(\underline{x},y)}].
\]
These are given by the isomorphisms
    \[
       (\gamma_{A,B})_{[p;a_1,\ldots,a_n],[q;b_1,\ldots,b_m]} = [\lambda_{p,q},\underline{1}],
    \]
    by which we mean the composite
    \[
    [\mu(p; \underline{q}); \underline{(x, \underline{y})}] = [\mu(p; \underline{q}) \cdot t_{m,n}; t_{m,n}^{-1} \cdot \underline{(x, \underline{y})}] \stackrel{[\lambda, 1]}{\longrightarrow} [\mu(q; \underline{p}); \underline{(\underline{x},y)}].
    \]
In the case that either $p$ or $q$ is an identity then we choose the component of $\gamma$ to be the isomorphism involving the appropriate identity element, as assumed to exist in the statement of the theorem.
There are two things to note about the definition above before we continue.  First, it is easy to check that
\[
t_{m,n}^{-1} \cdot \underline{(x, \underline{y})} = \underline{(\underline{x},y)}
\]
since $\pi(t_{m,n}) = \tau_{m,n}$.  Second, the morphism above has second component the identity.  This is actually forced upon us by the requirement that $\gamma$ be a modification:  in the case that $A,B$ are discrete categories, the only possible morphism is an identity, and the modification axiom then forces that statement to be true for general $A,B$ by considering the inclusion $A_{0} \times B_{0} \hookrightarrow A \times B$ where $A_{0}, B_{0}$ are the discrete categories with the same objects as $A, B$.

We show that this is a modification by noting that it does not rely on objects in the lists $a_1, \ldots, a_n$ or $b_1, \ldots, b_m$, only on their lengths and the operations $p$ and $q$. As a result, if we have functors $f \colon X \rightarrow X'$ and $g \colon Y \rightarrow Y'$, then it is clear that
    \[
        (\underline{P}(f\times g) \circ \gamma_{X,Y})_{[p;\underline{x}],[q;\underline{y}]} = [\lambda_{p,q},\underline{1}] = (\gamma_{X',Y'} \circ (\underline{P}f\times \underline{P}g))_{[p;\underline{x}],[q;\underline{y}]}.
    \]
As such we can simply write $(\gamma_{X,Y})_{[p;\underline{x}],[q;\underline{y}]}$ in shorthand as $\gamma_{p,q}$.

The three strength axioms are immediately satisfied, again since $\gamma_{p,q}$ has no dependence on the objects in the lists and as such the isomorphisms are the same. The unit axioms follow from the assumption that $t_{1,n} = e_n = t_{n,1}$ and that components of $\gamma$ involving an identity operation are also identity maps. The multiplication axioms follow from the two diagrams assumed to commute in the statement of the theorem.  If we consider each axiom to consist of two equations, one in $P(n)$ and one in some power  $(A \times B)^{n}$, then the two diagrams at the end of the statement of the theorem actually force the first components of the two multiplication axioms to hold in $P(n)$ before taking any equivalence classes in the coequalizer aside from the ones used to define $\gamma_{A,B}$ above..
\end{proof}

A further property that a pseudo-commutativity can possess is that of symmetry.  This symmetry is then reflected in the monoidal structure on the 2-category of algebras, which will then also have a symmetric tensor product (in a suitable, 2-categorical sense).

\begin{Defi}
Let $T \colon \m{K} \rightarrow \m{K}$ be a $2$-monad on a symmetric monoidal $2$-category $\m{K}$ with symmetry $c$. We then say that a pseudo-commutativity $\gamma$ for $T$ is \textit{symmetric} when the following is satisfied for all $A$, $B \in \m{K}$:
    \[
        Tc_{A,B} \circ \gamma_{A,B} \circ c_{TB, TA} = \gamma_{B,A}.
    \]
\end{Defi}

With the notion of symmetry at hand we are able to extend the above theorem, stating when $\underline{P}$ is symmetric.
\begin{thm}
The pseudo-commutative structure for $\underline{P}$ given by Theorem \ref{pscomm}  is symmetric if for all $m,n \in \mathbb{N}_+$ the two conditions below hold.
    \begin{enumerate}
        \item $t_{m,n} = t_{n,m}^{-1}$.
        \item The following diagram commutes:
            \[
                \xy
                    (0,0)*+{\mu(p;\underline{q}) \cdot t_{m,n}t_{n,m}}="00";
                    (30,0)*+{\mu(p;\underline{q}) \cdot e_{mn}}="10";
                    (0,-15)*+{\mu(q;\underline{p}) \cdot t_{n,m}}="01";
                    (30,-15)*+{\mu(p;\underline{q})}="11";
                    {\ar@{=} "00" ; "10"};
                    {\ar_{\lambda_{p,q} \cdot 1} "00" ; "01"};
                    {\ar@{=} "10" ; "11"};
                    {\ar_{\lambda_{q,p}} "01" ; "11"};
                \endxy
            \]
    \end{enumerate}
\end{thm}
\begin{proof}
The commutativity of the diagram above ensures that the first component of the symmetry axiom commutes in $P(n)$ before taking equivalence classes in the coequalizer, just as in Theorem \ref{pscomm}.
\end{proof}

\begin{Defi}
Let $P$ be a $\mb{G}$-operad in $\mb{Cat}$.  We say that $P$ is \textit{contractible} if each category $P(n)$ is equivalent to the terminal category.
\end{Defi}

\begin{cor}
If $P$ is contractible and there exist $t_{m,n}$ as in Theorem \ref{pscomm}, then $\underline{P}$ acquires a pseudo-commutativity. Furthermore, it is symmetric if $t_{n,m} = t_{m,n}^{-1}$.
\end{cor}
\begin{proof}
The only thing left to define is the collection of natural isomorphisms $\lambda_{p,q}$.  But since each $P(n)$ is contractible, $\lambda_{p,q}$ must be the unique isomorphism between its source and target, and furthermore the last two axioms hold since any pair of parallel arrows are equal in a contractible category.
\end{proof}

\begin{cor}
If $P$ is a contractible symmetric operad then $\underline{P}$ has a symmetric pseudo-commutativity.
\end{cor}
\begin{proof}
We choose $t_{m,n} = \tau_{m,n}$.
\end{proof}

\begin{cor}
Let $P$ be a non-symmetric operad. Then $\underline{P}$ is never pseudo-commutative.
\end{cor}
\begin{proof}
In the non-symmetric case, the 2-monad is just given using coproducts and products, there is no coequalizer. Thus when $A,B$ are discrete, there is no isomorphism $\underline{(x,\underline{y})} \cong \underline{(\underline{x},y)}$.
\end{proof}

We conclude with a computation using Theorem \ref{pscomm}.  This result was only conjectured in \cite{HP}, but we are able to prove it quite easily with the machinery developed thus far.

\begin{thm}\label{braidpscomm}
The 2-monad $\underline{B}$ for braided strict monoidal categories on $\mb{Cat}$ has two pseudo-commutative structures on it, neither of which are symmetric.
\end{thm}

In order to apply our theory, the 2-monad $\underline{B}$ must arise from a $\mb{G}$-operad.  The following proposition describes it as such, and can be found in the work of Fiedorowicz \cite{fie-br}.

\begin{prop}
The 2-monad $\underline{B}$ is the 2-monad associated to the $\mb{Br}$-operad $B$ with the category $B(n)$ having objects the elements of the $n$th braid group $Br_{n}$ and a unique isomorphism between any pair of objects; the action of $Br_{n}$ on $B(n)$ is given by right multiplication on objects and is then uniquely determined on morphisms.
\end{prop}

The interested reader can easily verify that algebras for the $\mb{Br}$-operad $B$ are braided strict monoidal categories.  The objects of $\underline{B}(X)$ can be identified with finite lists of objects of $X$, and morphisms are generated by the morphisms of $X$ together with new isomorphisms
\[
x_{1}, \ldots, x_{n} \stackrel{\gamma}{\longrightarrow} x_{\gamma^{-1}(1)}, \ldots, x_{\gamma^{-1}(n)}
\]
where $\gamma \in Br_{n}$ and the notation $\gamma^{-1}(i)$ means, as before, that we take the preimage of $i$ under the permutation $\pi(\gamma)$ associated to $\gamma$.  This shows that $\underline{B}(X)$ is the free braided strict monoidal category generated by $X$ according to \cite{js}, and it is easy to verify that the 2-monad structure on $\underline{B}$ arising from the $\mb{Br}$-operad structure on $B$ is the correct one to produce braided strict monoidal categories as algebras.

\begin{Defi}
A braid $\gamma \in Br_{n}$ is \textit{positive} if it is an element of the submonoid of $Br_{n}$ generated by the elements $\sigma_{1}, \sigma_{2}, \ldots, \sigma_{n-1}$.
\end{Defi}

\begin{Defi}
 A braid $\gamma \in Br_{n}$ is \textit{minimal} if no pair of strands in $\gamma$ cross twice.
\end{Defi}

For our purposes, we are interested in braids which are both positive and minimal.  A proof of the following lemma can be found in \cite{EM2}.

\begin{lem}\label{pmlem}
Let $PM_{n}$ be the subset of $Br_{n}$ consisting of positive, minimal braids.  Then the function sending a braid to its underlying permutation is a bijection of sets $PM_{n} \rightarrow \Sigma_{n}$.
\end{lem}

\begin{rem}\label{pmrem}
It is worth noting that this bijection is not an isomorphism of groups, since $PM_{n}$ is not a group or even a monoid.  The element $\sigma_{1} \in Br_{n}$ is certainly in $PM_{n}$, but $\sigma_{1}^{2}$ is not as the first two strands cross twice.  Thus we see that the product of two minimal braids is generally not minimal, but by definition the product of positive braids is positive.
\end{rem}

\begin{proof}[Proof of Theorem \ref{braidpscomm}]
In order to use Theorem \ref{pscomm} with the action operad being the braid operad $\mb{Br}$, we must first construct elements $t_{m,n} \in Br_{mn}$ satisfying certain properties.  Using Lemma \ref{pmlem}, we define $t_{m,n}$ to be the unique positive braid such that $\pi(t_{m,n}) = \tau_{m,n}$.  Since $\tau_{1,n} = e_{n} = \tau_{n,1}$ in $\Sigma_{n}$ and the identity element $e_{n} \in Br_{n}$ is positive and minimal, we have that $t_{1,n} = e_{n} = t_{n,1}$ in $Br_{n}$.  Thus in order to verify the remaining hypotheses, we must check two equations, each of which states that some element $t_{m,n}$ can be expressed as a product of operadic compositions of other elements.

Let $l, m_{1}, \ldots, m_{l}, n$ be natural numbers, and let $M = \sum m_{i}$.  We must check that
\[
\mu(e_{l}; t_{n, m_{1}}, \ldots, t_{n, m_{l}}) \mu(t_{n,l}; \underline{e_{m_{1}}, \ldots, e_{m_{l}}}) = t_{N, l}
\]
in $Br_{lN}$.  These braids have the same underlying permutations by construction, and both are positive since each operadic composition on the left is positive.  The braid on the right is minimal by definition, so if we prove that the braid on the left is also minimal, they are necessarily equal.  Now $\mu(t_{n,l}; \underline{e_{m_{1}}, \ldots, e_{m_{l}}})$ is given by the braid for $t_{n,l}$ but with the first strand replaced by $m_{1}$ strands, the second strand replaced by $m_{2}$ strands, and so on for the first $l$ strands of $t_{n,l}$, and then repeating for each group of $l$ strands.  In particular, since strands $i, i+l, i+2l, \ldots, i + (n-1)l$ never cross in $t_{n,l}$, none of the $m_{i}$ strands that each of these is replaced with cross.  The braid $\mu(e_{l}; t_{n, m_{1}}, \ldots, t_{n, m_{l}})$ consists of the disjoint union of the braids for each $t_{n,m_{i}}$, so if two strands cross in $\mu(e_{l}; t_{n, m_{1}}, \ldots, t_{n, m_{l}})$ then they must both cross in some $t_{n,m_{i}}$.  The strands in $t_{n,m_{i}}$ are those numbered from $n(m_{1} + \cdots + m_{i-1}) + 1$ to $n(m_{1} + \cdots + m_{i-1} + m_{i})$.  This is a consecutive collection of $nm_{i}$ strands, and it is simple to check that these strands are precisely those connected (via the group operation in $Br_{Nl}$, concatenation) to the duplicated copies of strands $i, i+l, i+2l, \ldots, i + (n-1)l$ in $t_{n,l}$.  Thus if a pair of strands were to cross in $\mu(e_{l}; t_{n, m_{1}}, \ldots, t_{n, m_{l}})$, that pair cannot also have crossed in $\mu(t_{n,l}; \underline{e_{m_{1}}, \ldots, e_{m_{l}}})$, showing that the resulting product braid
\[
\mu(e_{l}; t_{n, m_{1}}, \ldots, t_{n, m_{l}}) \mu(t_{n,l}; \underline{e_{m_{1}}, \ldots, e_{m_{l}}})
\]
is minimal.  The calculation showing that
\[
\mu(t_{m,l}; \underline{e_{1}}, \ldots, \underline{e_{n_{m}}}) \mu(e_{m}; t_{n_{1}, l}, \ldots, t_{n_{m}, l})
\]
is also minimal follows from the same argument, showing that it is equal to $t_{N, l}$ (here $N$ is the sum of the $n_{i}$, where once again $i$ ranges from 1 to $l$).

These calculations show, using Theorem \ref{pscomm}, that the $\mb{Br}$-operad $B$ induces a 2-monad which has a pseudo-commutative structure.  As noted before, $B$-algebras are precisely braided strict monoidal categories.  The second pseudo-commutative structure arises by using negative, minimal braids instead of positive ones, and proceeds using the same arguments.  This finishes the first part of the proof of Theorem \ref{braidpscomm}.

We will now show that neither of these pseudo-commutative structures is symmetric.  The symmetry axiom in this case reduces to the fact that, in some category which is given as a coequalizer, the morphism with first component
\[
f:\mu(p; \underline{q}) \cdot t_{n,m}t_{m,n} \rightarrow \mu(q; \underline{p}) \cdot t_{m,n} \rightarrow \mu(p; \underline{q})
\]
is the identity.  Now it is clear that that $t_{n,m}$ is not equal to $t_{m,n}^{-1}$ in general: taking $m=n=2$, we note that $t_{2,2} = \sigma_{2}$, and this element is certainly not of order two in $Br_{4}$.  $B(4)$ is the category whose objects are the elements of $Br_{4}$ with a unique isomorphism between any two pair of objects, and $Br_{4}$ acts by multiplication on the right; this action is clearly free and transitive.  We recall (see Lemma \ref{coeq-lem}) that in a coequalizer of the form $A \times_{G} B$, we have that a morphism $[f_{1}, f_{2}]$ equals $[g_{1}, g_{2}]$ if and only if there exists an $x \in G$ such that
\[
\begin{array}{rcl}
f_{1} \cdot x & = & g_{1}, \\
x^{-1} \cdot f_{2} & = & g_{2}.
\end{array}
\]
For the coequalizer in question, for $f$ to be the first component of an identity morphism would imply that $f \cdot x$ would be a genuine identity in $B(4)$ for some $x$.  But $f \cdot x$ would have source $\mu(p; \underline{q}) t_{n,m}t_{m,n}x$ and target $\mu(p; \underline{q})x$, which requires $t_{n,m}t_{m,n}$ to be the identity group element for all $n,m$.  In particular, this would force $t_{2,2}$ to have order two, which we have noted above does not hold in $Br_{4}$, thus giving a contradiction.
\end{proof}

\begin{rem}
The pseudo-commutativities given above are not necessarily the only ones that exist for the $\mb{Br}$-operad $B$, but we do not know a general strategy for producing others.
\end{rem}

\bibliographystyle{plain}
\bibliography{operad_ref}

\end{document}